\documentclass{amsart}
\usepackage{amssymb}
\usepackage{amscd}
\usepackage{a4wide}
\usepackage[all]{xy}
\usepackage[utf8]{inputenc}

\usepackage{color,ulem}

\newcommand{\cC}{\mathcal{C}}
\newcommand{\cM}{\mathcal{M}}

  \newtheorem{proposition}{Proposition}[section]

  \newtheorem{theorem}[proposition]{Theorem}

  \theoremstyle{definition}
  \newtheorem{definition}[proposition]{Definition}
   
  \newtheorem{example}[proposition]{Example}

  \theoremstyle{remark}
  \newtheorem{remark}[proposition]{Remark}

\begin{document}

\title{Hom-entwining structures and Hom-Hopf-type modules}
\author{SERKAN KARA\c{C}UHA}

\address{Department of Mathematics, FCUP, University of Porto, Rua Campo Alegre
687, 4169-007 Porto, Portugal}
\email{s.karacuha34@gmail.com}

\begin{abstract}The notions of Hom-coring, Hom-entwining structure and associated entwined Hom-module are introduced. A theorem regarding base ring extension of a Hom-coring is proven and then is used to acquire the Hom-version of Sweedler coring \cite{Sweedler2}. Motivated by \cite{Brzezinski2}, a Hom-coring associated to an entwining Hom-structure is constructed and an identification of entwined Hom-modules with Hom-comodules of this Hom-coring is shown. The dual algebra of this Hom-coring is proven to be a $\psi$-twisted convolution algebra. By a construction, it is shown that a Hom-Doi-Koppinen datum comes from a Hom-entwining structure and that the Doi-Koppinen Hom-Hopf modules are the same as the associated entwined Hom-modules, following \cite{Brzezinski1}. A similar construction regarding an alternative Hom-Doi-Koppinen datum is also given. A collection of Hom-Hopf-type modules are gathered as special examples of Hom-entwining structures and corresponding entwined Hom-modules, and structures of all relevant Hom-corings are also considered.
\end{abstract}

\maketitle
\section{Introduction}
Motivated by the study of symmetry properties of noncommutative principal bundles, {\it entwining structures} (over a commutative ring $k$ with unit) were introduced in \cite{BrzezinskiMajid} as a triple $(A,C)_{\psi}$ consisting of a $k$-algebra $A$, a $k$-coalgebra $C$ and a $k$-module map $\psi:C\otimes A \to A\otimes C$ satisfying, for all $a, a'\in A$ and $c\in C$,
$$(aa')_{\kappa}\otimes c^{\kappa}=a_{\kappa}a'_{\lambda}\otimes c^{\kappa\lambda}\:,\:\: 1_{\kappa}\otimes c^{\kappa}=1\otimes c,$$
$$a_{\kappa}\otimes c^{\kappa}_{\ 1}\otimes c^{\kappa}_{\ 2}=a_{\lambda\kappa}\otimes c_{1}^{\ \kappa}\otimes c_{2}^{\ \lambda}\:,\:\: a_{\kappa}\varepsilon(c^{\kappa})=a\varepsilon(c), $$
where the notation $\psi(c\otimes a)=a_{\kappa}\otimes c^{\kappa}$ (summation over $\kappa$ is understood) is used. Given an entwining structure $(A,C)_{\psi}$, the notion of $(A,C)_{\psi}$-entwined module $M$ was first defined in \cite{Brzezinski1} as a right $A$-module with action $m\otimes a\mapsto m\cdot a $ and a right $C$-comodule with coaction $\rho^{M}:m\mapsto m_{(0)}\otimes m_{(1)}$ (summation understood) such that the following compatibility condition holds:
$$\rho^{M}(m\cdot a)=m_{(0)}\cdot a_{\kappa}\otimes m_{(1)}^{\ \ \ \kappa}\:,\: \: \forall a\in A, m\in M.$$

{\it Hopf-type modules} are typically the objects with an action of an algebra and a coaction of a coalgebra which satisfy some compatibility condition. The family of Hopf-type modules includes well known examples such as Hopf modules of Sweedler \cite{Sweedler1}, relative Hopf modules of Doi and Takeuchi \cite{Doi1}, \cite{Takeuchi}, Long dimodules \cite{Long}, Yetter-Drinfeld modules \cite{RadfordTowber}, \cite{Yetter}, Doi-Koppinen Hopf modules \cite{Doi2}, \cite{Koppinen} and alternative Doi-Koppinen Hopf modules of Schauenburg \cite{Schauenburg}. All these modules except alternative Doi-Koppinen modules are special cases of Doi-Koppinen modules. As newer special cases of them, the family of Hopf-type modules also includes anti-Yetter-Drinfeld modules which were obtained as coefficients for the cyclic cohomology of Hopf algebras \cite{HajacKhalkhaliRangipourSommerhauser1}, \cite{HajacKhalkhaliRangipourSommerhauser2}, \cite{JaraStefan}, and their generalizations termed $(\alpha, \beta)$-Yetter-Drinfeld modules \cite{PanaiteStaic} (also called $(\alpha, \beta)$-equivariant $C$-comodules in \cite{KaygunKhalkhali}). Basically, entwining structures and modules associated to them enable us to unify several categories of Hopf modules in the sense that the compatibility conditions for all of them can be restated in the form of the above condition required for entwined modules. One can refer to \cite{BrzezinskiWisbauer} and \cite{CaenepeelMilitaruZhu} for more information on the relationship between entwining structures and Hopf-type modules.

Hom-type algebras have been introduced in the form of Hom-Lie algebras in \cite{HartwigLarssonSilvestrov}, where the Jacobi identity was twisted along a linear endomorphism. Meanwhile, Hom-associative algebras have been suggested in \cite{MakhloufSilvestrov} to give rise to a Hom-Lie algebra using the commutator bracket. Other Hom-type structures such as Hom-coalgebras, Hom-bialgebras, Hom-Hopf algebras and their properties have been considered in \cite{MakhloufSilvestrov1,MakhloufSilvestrov2,Yau}. The so-called \textit{twisting principle} has been introduced in \cite{Yau1} to provide Hom-type generalization of algebras, and it has been used to obtain many more properties of Hom-bialgebras  and Hom-Hopf algebras; for instance see \cite{AmmarMakhlouf,FregierGohr,Gohr,MakhloufPanaite,Yau2,Yau5}. The authors of \cite{CaenepeelGoyvaerts} investigated the counterparts of Hom-bialgebras and Hom-Hopf algebras in the context of tensor categories, and termed them {\it monoidal} Hom-bialgebras and monoidal Hom-Hopf algebras with slight variations in their definitions. Further properties of monoidal Hom-Hopf algebras and many structures on them have been lately studied \cite{ChenWangZhang},\cite{ChenWangZhang1},\cite{ChenZhang}, \cite{ChenZhang1}, \cite{GuoChen}, \cite{Karacuha}, \cite{LiuShen}.

Entwining structures have been generalized to weak entwining structures in \cite{CaenepeelDeGroot} to define Doi-Koppinen data for a weak Hopf algebra, motivated by \cite{Bohm}. Thereafter, it has been proven in \cite{Brzezinski2} that both entwined modules and weak entwined modules are comodules of certain type of corings which built on a tensor product of an algebra and a coalgebra, and shown that various properties of entwined modules can be obtained from properties of comodules of a coring. Here we recall from \cite{Sweedler2} that for an associative algebra $A$ with unit,  an $A$-{\it coring} is an $A$-bimodule $\mathcal{C}$ with $A$-bilinear maps $\Delta_{\mathcal{C}}:\mathcal{C} \to \mathcal{C}\otimes_A\mathcal{C},\: c\mapsto c_1\otimes c_2$ called coproduct and $\varepsilon_{\mathcal{C}}:\mathcal{C}\to A$ called counit, such that
$$\Delta_{\mathcal{C}}(c_1)\otimes c_2=c_1\otimes \Delta_{\mathcal{C}}(c_2),\: \varepsilon_{\mathcal{C}}(c_1)c_2=c=c_1\varepsilon_{\mathcal{C}}(c_2),\: \forall c \in \mathcal{C}. $$
Given an $A$-coring $\mathcal{C}$, a {\it right} $\mathcal{C}$-{\it comodule} is a right $A$-module $M$ with a right $A$-linear map $\rho^{M}:M\to M\otimes \mathcal{C},\: m\mapsto m_{(0)}\otimes m_{(1)}$ called coaction, such that
$$\rho^{M}(m_{(0)})\otimes m_{(1)}=m_{(0)}\otimes \Delta_{\mathcal{C}}(m_{(1)}), \: m=m_{(0)}\varepsilon_{\mathcal{C}}(m_{(1)}),\: \forall m\in M.$$

The main aim of the present paper is to generalize the entwining structures, entwined modules and the associated corings within the framework of monoidal Hom-structures and then to study Hopf-type modules in the Hom-setting. The idea is to replace algebra and coalgebra in a classical entwining structure with a monoidal Hom-algebra and a monoidal Hom-coalgebra to make a definition of Hom-entwining structures and associated entwined Hom-modules. Following \cite{Brzezinski2}, these entwined Hom-modules are identified with Hom-comodules of the associated Hom-coring. The dual algebra of this Hom-coring is proven to be the Koppinen smash. Furthermore, we give a construction regarding Hom-Doi-Kopinen datum and Doi-Koppinen Hom-Hopf modules as special cases of Hom-entwining structures and associated entwined Hom-modules. Besides, we introduce alternative Hom-Doi-Koppinen datum. By using these constructions, we get Hom-versions of the aforementioned Hopf-type modules as special cases of entwined Hom-modules, and give examples of Hom-corings in addition to trivial Hom-coring and canonical Hom-coring.

Throughout the paper $k$ will be a commutative ring with unit. Unadorned tensor product is over $k$.

\section{Preliminaries}
Let $\cM_k=(\cM_k,\otimes,k,a,l,r)$ be the monoidal category of $k$-modules. We associate to $\cM_k$ a new monoidal category $\mathcal{H}(\cM_k)$ whose objects are ordered pairs $(M,\mu)$,  with $M\in\mathcal{M}_k$ and $\mu\in Aut_k(M)$, and morphisms $f:(M,\mu)\to(N,\nu)$ are morphisms $f:A\to B$ in $\mathcal{M}_k$ satisfying $\nu\circ f=f\circ\mu.$ The monoidal structure is given by $(M,\mu)\otimes(N,\nu)=(M\otimes N,\mu\otimes\nu)$ and $(k,1)$. If we speak briefly, all monoidal Hom-structures are objects in the tensor category $\mathcal{\widetilde{H}}(\cM_k)=(\mathcal{H}(\cM_k),\otimes,(k,id),\tilde{a},\tilde{l},\tilde{r})$ introduced in (\cite{CaenepeelGoyvaerts}), with the associativity constraint $\tilde{a}$ defined by
 \begin{equation}\tilde{a}_{A,B,C}=a_{A,B,C}\circ((\alpha\otimes id)\otimes \gamma^{-1})=(\alpha\otimes (id\otimes \gamma^{-1}))\circ a_{A,B,C},\end{equation} for $(A,\alpha)$, $(B,\beta)$, $(C,\gamma)$ $\in$ $\mathcal{H}(\cM_k)$, and the right and left unit constraints $\tilde{r}$, $\tilde{l}$ given by
  \begin{equation}\tilde{r}_A=\alpha\circ r_A=r_A\circ(\alpha\otimes id);\:\: \tilde{l}_A=\alpha\circ l_A=l_A\circ(id\otimes\alpha), \end{equation}
which we write elementwise:
$$\tilde{a}_{A,B,C}((a\otimes b)\otimes c)=\alpha(a)\otimes(b\otimes\gamma^{-1}(c)), $$
$$\tilde{l}_A(x\otimes a)=x\alpha(a)=\tilde{r}_A(a\otimes x).$$
The category $\mathcal{\widetilde{H}}(\cM_k)$ is termed {\it Hom-category} associated to $\mathcal{M}_k$, where a $k$-submodule $N\subset M$ is called a subobject of $(M,\mu)$ if $(N,\mu_{|_N})$ $\in$ $\mathcal{\widetilde{H}}(\cM_k)$, that is $\mu$ restricts to an automorphism of $N$.

We now recall some definitions of monoidal Hom-structures from \cite{CaenepeelGoyvaerts}. An algebra in $\widetilde{\mathcal{H}}(\mathcal{M}_k)$ is called a monoidal Hom-algebra, i.e., an object $(A,\alpha) \in \widetilde{\mathcal{H}}(\mathcal{M}_k)$ together with a $k$-linear map $m:A\otimes A\to A,\: a\otimes b\mapsto ab$ and an element $1_A \in A$ such that
\begin{equation}\label{monoidal-Hom-alg-cond-1}\alpha(a)(bc)=(ab)\alpha(c)\:\: ;\:\: a1_A=1_Aa=\alpha(a),\end{equation}
\begin{equation}\label{monoidal-Hom-alg-cond-2}\alpha(ab)=\alpha(a)\alpha(b)\:\: ;\:\: \alpha(1_A)=1_A\end{equation}
for all $a,b,c \in A$.  A right $(A,\alpha)$-Hom-module consists of an object $(M,\mu) \in \widetilde{\mathcal{H}}(\mathcal{M}_k)$ together with a $k$-linear map $\psi:M\otimes A\to M,\: \psi(m\otimes a)=ma$, in $\widetilde{\mathcal{H}}(\mathcal{M}_k)$ satisfying the following
   \begin{equation}\label{Hom-action-cond-1}\mu(m)(ab)= (ma)\alpha(b)\:; \: m1_A=\mu(m)\end{equation}
   for all $m\in M$ and $a, b \in A$. For $\psi$ to be a morphism in $\widetilde{\mathcal{H}}(\mathcal{M}_k)$ means
   \begin{equation}\label{Hom-action-cond-2}\mu(ma)= \mu(m)\alpha(a).\end{equation} The map  $\psi$ is termed a right Hom-action of $(A,\alpha)$ on $(M,\mu)$. Let $(M,\mu)$ and $(N,\nu)$ be two right $(A,\alpha)$-Hom-modules. We call a morphism $f:M\to N$ right $(A,\alpha)$-linear if $f\circ\mu=\nu\circ f$ and $f(ma)=f(m)a$ for all $m \in M$ and $a\in A$.

 Let $(A,\alpha)$ and $(B,\beta)$ be two monoidal Hom-algebras. A {\it left} $(A,\alpha)$, {\it right} $(B,\beta)$ {\it Hom-bimodule} (for short $[(A,\alpha),(B,\beta)]$-Hom-bimodule), consists of an object $(M,\mu) \in \widetilde{\mathcal{H}}(\mathcal{M}_k)$ together with a left $(A,\alpha)$-Hom-action $\phi:A\otimes M\to M$, $\phi(a\otimes m)=am$ and a right $(B,\beta)$-Hom-action $\varphi:M\otimes B\to M$, $\varphi(m\otimes b)=mb$ fulfilling the compatibility condition, for all $a\in A$, $b\in B$ and $m\in M$,
\begin{equation}(am)\beta(b)=\alpha(a)(mb). \end{equation}
Let $(M,\mu)$ and $(N,\nu)$ be two $[(A,\alpha),(B,\beta)]$-Hom-bimodules. A morphism $f:M\to N$ is called a morphism of $[(A,\alpha),(B,\beta)]$-Hom-bimodules if it is both left $(A,\alpha)$-linear and right $(B,\beta)$-linear, and satisfies the following property
\begin{equation}(af(m))\beta(b)=\alpha(a)(f(m)b),\end{equation}
for all $a\in A$, $b\in B$ and $m\in M$. Of course, any monoidal Hom-algebra $(A,\alpha)$ is a $[(A,\alpha),(A,\alpha)]$-Hom-bimodule ($(A,\alpha)$-Hom-bimodule, for short). Moreover, if $f:A\rightarrow B$ is a morphism of monoidal Hom-algebras $(A,\alpha)$ and $(B,\beta)$, then naturally $(B,\beta)$ is a $[(A,\alpha),(A,\alpha)]$-Hom-bimodule along $f$, i.e., the left and right $(A,\alpha)$-Hom-action on $(B,\beta)$ is given via $f$.

Let $(M,\mu)$ be a right $(A,\alpha)$-Hom-module and $(N,\nu)$ be a left $(A,\alpha)$-Hom-module. The tensor product $(M\otimes_A N, \mu\otimes\nu)$ of $(M,\mu)$ and $(N,\nu)$ over $(A,\alpha)$ is the coequalizer of
$\rho\otimes id_N,\: (id_M\otimes \bar{\rho})\circ \tilde{a}_{M,A,N}: (M\otimes A)\otimes N\to M\otimes N$, where $\rho$ and $\bar{\rho}$ are the right and left Hom-actions of $(A,\alpha)$ on $(M,\mu)$ and $(N,\nu)$ respectively. That is,
\begin{equation}\label{tensor-prod-over-Hom-alg} m\otimes_A n=\{m\otimes n \in M\otimes N| \: ma\otimes n=\mu(m)\otimes a\nu^{-1}(n),\forall a \in A\}.\end{equation}

\section{Hom-corings and Hom-Entwining structures}

Analogous with monoidal Hom-algebras, one defines monoidal Hom-coalgebras as coalgebras in  $\widetilde{\mathcal{H}}(\mathcal{M}_k)$. More generally we define the notion of a Hom-coring.
\begin{definition}
 \begin{enumerate}
\item Let $(A,\alpha)$ be a monoidal Hom-algebra. An $(A,\alpha)$-{\it Hom-coring} consists of an $(A,\alpha)$-Hom-bimodule $(\mathcal{C},\chi)$ together with $A$-bilinear maps $\Delta_{\mathcal{C}}: \mathcal{C}\to \mathcal{C}\otimes_A\mathcal{C}$, $c\mapsto c_1\otimes c_2$ and $\varepsilon_{\mathcal{C}}:\mathcal{C}\to A $ called {\it comultiplication} and {\it counit} such that

\begin{equation}\label{Hom-coring-cond-1}\chi^{-1}(c_1)\otimes \Delta_{\mathcal{C}}(c_2)=c_{11}\otimes(c_{12}\otimes\chi^{-1}(c_2));\ \ \varepsilon_{\mathcal{C}}(c_1)c_2=c=c_1\varepsilon_{\mathcal{C}}(c_2),
\end{equation}

\begin{equation}\label{Hom-coring-cond-2}\Delta_{\mathcal{C}}(\chi(c))=\chi(c_1)\otimes \chi(c_2);\ \ \varepsilon_{\mathcal{C}}(\chi(c))=\alpha(\varepsilon_{\mathcal{C}}(c)).
\end{equation}
\item A {\it right $(\mathcal{C},\chi)$-Hom-comodule } $(M,\mu)$ is defined as a right $(A,\alpha)$-Hom-module with a right $A$-linear map $\rho: M\to M\otimes_A \mathcal{C}$, $m\mapsto m_{(0)}\otimes m_{(1)}$ satisfying
    \begin{equation}\label{Hom-coaction-cond-1}\mu^{-1}(m_{(0)})\otimes \Delta_{\mathcal{C}}(m_{(1)})=m_{(0)(0)}\otimes (m_{(0)(1)}\otimes \chi^{-1}(m_{(1)}));\ \ m=m_{(0)}\varepsilon_{\mathcal{C}}(m_{(1)}),
    \end{equation}
    \begin{equation}\label{Hom-coaction-cond-2}\mu(m)_{(0)}\otimes \mu(m)_{(1)}=\mu(m_{(0)})\otimes\chi(m_{(1)}).
    \end{equation}
\end{enumerate}
\end{definition}

A monoidal Hom-coalgebra $(C,\gamma)$ is simply a coalgebra in the category $\widetilde{\mathcal{H}}(\mathcal{M}_k)$. It is not hard to see that  $(C,\gamma)$ is a monoidal Hom-coalgebra if and only if it is a $(k,id)$-Hom-coring, i.e. an object $(C,\gamma) \in \widetilde{\mathcal{H}}(\mathcal{M}_k)$ together with $k$-linear maps $\Delta:C\to C\otimes C,\:\Delta(c)=c_1\otimes c_2$ and $\varepsilon :C\to k$ satisfying equations (\ref{Hom-coring-cond-1}) and (\ref{Hom-coring-cond-2}).
A right Hom-comodule $(M,\mu)$ over a monoidal Hom-coalgebra $(C,\gamma)$ is a right $(C,\gamma)$-Hom-comodule over the $(k,id)$-Hom-coring $(C,\gamma)$, i.e., an object $(M,\mu) \in \widetilde{\mathcal{H}}(\mathcal{M}_k)$ together with a $k$-linear map $\rho:M\to M\otimes C,\: \rho(m)=m_{[0]}\otimes m_{[1]}$, in $\widetilde{\mathcal{H}}(\mathcal{M}_k)$ such that (\ref{Hom-coaction-cond-1}) and (\ref{Hom-coaction-cond-2}) hold. The map $\rho$ is called a right Hom-coaction of $(C,\gamma)$ on $(M,\mu)$. Let $(M,\mu)$ and $(N,\nu)$ be two right $(C,\gamma)$-Hom-comodules, then we call a morphism $f:M\to N$ right $(C,\gamma)$-colinear if $f\circ \mu=\nu\circ f$ and $f(m_{[0]})\otimes m_{[1]}=f(m)_{[0]}\otimes f(m)_{[1]}$ for all $m \in M$.

\begin{theorem}\label{base-ring-extension-of-Hom-coring}Let $\phi:(A,\alpha)\to (B,\beta)$ be a morphism of monoidal Hom-algebras. Then, for an $(A,\alpha)$-Hom-coring $(\mathcal{C},\chi)$,  $(B\mathcal{C})B=((B\otimes_{A}\mathcal{C})\otimes_{A}B,(\beta\otimes\chi)\otimes\beta)$ is a $(B,\beta)$-Hom-coring, called a {\it base ring extension} of the $(A,\alpha)$-Hom-coring $(\cC,\chi)$, with a comultiplication and a counit,
\begin{equation}\Delta_{(B\mathcal{C})B}((b\otimes_A c)\otimes_A b')=((\beta^{-1}(b)\otimes_A c_1)\otimes_A 1_B)\otimes_B ((1_B\otimes_A c_2)\otimes_A\beta^{-1}(b')),
\end{equation}
\begin{equation}\varepsilon_{(B\mathcal{C})B}((b\otimes_A c)\otimes_A b')=(b\phi(\varepsilon_{\mathcal{C}}(c)))b'.\end{equation}
\end{theorem}

\begin{proof} For $b,b',b''\in B$ and $c\in \cC$,
\begin{eqnarray*}\lefteqn{\Delta_{(B\mathcal{C})B}(((b'\otimes_A c)\otimes_A b'')b)}\hspace{4em}\\
&=&\Delta_{(B\mathcal{C})B}((\beta(b')\otimes_A\chi(c))\otimes_A b''\beta^{-1}(b))\\
&=&((b'\otimes_A \chi(c)_1)\otimes_A 1_B)\otimes_B((1_B \otimes_A \chi(c)_2)\otimes_A\beta^{-1}(b''\beta^{-1}(b)))\\
&\overset{(\ref{Hom-coring-cond-2})}{=}&((b'\otimes_A \chi(c_1))\otimes_A 1_B)\otimes_B((1_B \otimes_A \chi(c_2))\otimes_A\beta^{-1}(b'')\beta^{-2}(b))\\
&=&((b'\otimes_A \chi(c_1))\otimes_A 1_B)\otimes_B((\beta\otimes\chi)(1_B \otimes_A c_2)\otimes_A\beta^{-1}(b'')\beta^{-1}(\beta^{-1}(b)))\\
&=&((\beta\otimes\chi)\otimes\beta)((\beta^{-1}(b')\otimes_A c_1)\otimes_A 1_B)\otimes_B((1_B\otimes_A c_2)\otimes_A\beta^{-1}(b''))\beta^{-1}(b)\\
&=&\Delta_{(B\mathcal{C})B}((b'\otimes_A c)\otimes_A b'')b,
\end{eqnarray*}
which proves the right $(B,\beta)$-linearity of $\Delta_{(B\mathcal{C})B}$. It can also be shown that $\Delta_{(B\mathcal{C})B}\circ \bar{\chi}=(\bar{\chi}\otimes\bar{\chi})\circ\Delta_{(B\mathcal{C})B}$, where $\bar{\chi}=(\beta\otimes\chi)\otimes\beta$. And as well, the left $(B,\beta)$-linearity of $\Delta_{(B\mathcal{C})B}$ and the fact that it preserves the compatibility condition between the left and right $(B,\beta)$-Hom-actions on $(B\cC)B$ can be checked similarly, that is,
$$\Delta_{(B\mathcal{C})B}(b((b'\otimes_A c)\otimes_A b''))=b\Delta_{(B\mathcal{C})B}((b'\otimes_A c)\otimes_A b''),$$
$$(b\Delta_{(B\mathcal{C})B}((b''\otimes_A c)\otimes_A b'''))\beta(b')=\beta(b)(\Delta_{(B\mathcal{C})B}((b''\otimes_A c)\otimes_A b''')b').$$
Next we prove the Hom-coassociativity of $\Delta_{(B\mathcal{C})B}$:

\begin{eqnarray*}\lefteqn{((\beta^{-1}\otimes \chi^{-1})\otimes \beta^{-1})(((b\otimes_A c)\otimes_A b')_1)\otimes_B\Delta_{(B\mathcal{C})B}(((b\otimes_Ac)\otimes_Ab')_2)}\hspace{6em}\\
&=&((\beta^{-2}(b)\otimes_A \chi^{-1}(c_1))\otimes_A1_B)\otimes_B(((1_B\otimes_A c_{21})\otimes_A 1_B)\\
&&\otimes_B((1_B\otimes_A c_{22})\otimes_A\beta^{-2}(b')))\\
&\overset{(\ref{Hom-coring-cond-1})}{=}&((\beta^{-2}(b)\otimes_A c_{11})\otimes_A1_B)\otimes_B(((1_B\otimes_A c_{12})\otimes_A 1_B)\\
&&\otimes_B((1_B\otimes_A \chi^{-1}(c_2))\otimes_A\beta^{-2}(b')))\\
&=&((\beta^{-1}(b)\otimes_A c_1)\otimes_A1_B)_1\otimes_B(((\beta^{-1}(b)\otimes_A c_1)\otimes_A1_B)_2\\
&&\otimes_B((1_B\otimes_A \chi^{-1}(c_2))\otimes_A\beta^{-2}(b')))\\
&=&((b\otimes_A c)\otimes_A b')_{11}\otimes_B(((b\otimes_A c)\otimes_A b')_{12}\\
&&\otimes_B((\beta^{-1}\otimes \chi^{-1})\otimes \beta^{-1})((b\otimes_A c)\otimes_A b')_2).
\end{eqnarray*}
Now we demonstrate that $\varepsilon_{(B\mathcal{C})B}$ is left $(B,\beta)$-linear:

\begin{eqnarray*}\lefteqn{\varepsilon_{(B\mathcal{C})B}(b((b'\otimes_A c)\otimes_A b''))}\hspace{4em}\\
&=&\varepsilon_{(B\mathcal{C})B}((\beta^{-2}(b)b'\otimes_A \chi(c))\otimes_A\beta(b''))\\
&=&((\beta^{-2}(b)b')\phi(\varepsilon_{\cC}(\chi(c))))\beta(b'')
\overset{(\ref{Hom-coring-cond-2})}{=}((\beta^{-2}(b)b')\phi(\alpha(\varepsilon_{\cC}(c))))\beta(b'')\\
&=&(\beta^{-1}(b)(b'\beta^{-1}(\phi(\alpha(\varepsilon_{\cC}(c))))))\beta(b'')=(\beta^{-1}(b)(b'\phi(\varepsilon_{\cC}(c))))\beta(b'')\\
&=&b((b'\phi(\varepsilon_{\cC}(c)))b'')=b\varepsilon_{(B\mathcal{C})B}((b'\otimes_A c)\otimes_A b''),
\end{eqnarray*}
where $\phi\circ\alpha=\beta\circ\phi$ was used in the fifth equality. Additionally, we have

\begin{eqnarray*}(\varepsilon_{(B\mathcal{C})B}\circ\bar{\chi})((b\otimes_A c)\otimes_A b')&=&(\beta(b)\phi(\varepsilon_{\cC}(\chi(c))))\beta(b')\\
&=&\beta((b\phi(\varepsilon_{\cC}(c)))b')=(\beta\circ\varepsilon_{\cC})((b\otimes_A c)\otimes_A b'),
\end{eqnarray*}
meaning $\varepsilon_{(B\mathcal{C})B}\in\mathcal{\widetilde{H}}(\cM_k)$. In the same manner, one can show that $\varepsilon_{(B\mathcal{C})B}$ is right $(B,\beta )$-linear and it preserves the compatibility condition between the left and right $(B,\beta)$-Hom-actions on $(B\cC)B$, i.e.,
$$\varepsilon_{(B\mathcal{C})B}(((b'\otimes_A c)\otimes_A b'')b)=\varepsilon_{(B\mathcal{C})B}((b'\otimes_A c)\otimes_A b'')b,$$
$$(b\varepsilon_{(B\mathcal{C})B}((b''\otimes_A c)\otimes_A b'''))\beta(b')=\beta(b)(\varepsilon_{(B\mathcal{C})B}((b''\otimes_A c)\otimes_A b''')b').$$
Below, we prove the counity condition:
\begin{eqnarray*}\lefteqn{((\beta^{-1}(b)\otimes_A c_1)\otimes_A 1_B)\varepsilon_{(B\mathcal{C})B}((1_B\otimes_A c_2)\otimes_A\beta^{-1}(b'))}\hspace{8em}\\
&=&((\beta^{-1}(b)\otimes_A c_1)\otimes_A 1_B)((1_B\phi(\varepsilon_{\cC}(c_2)))\beta^{-1}(b'))\\
&=&((\beta^{-1}(b)\otimes_A c_1)\otimes_A 1_B)(\beta(\phi(\varepsilon_{\cC}(c_2)))\beta^{-1}(b'))\\
&=&(b\otimes_A \chi(c_1))\otimes_A \beta(\phi(\varepsilon_{\cC}(c_2)))\beta^{-1}(b')\\
&=&(b\otimes_A \chi(c_1))\otimes_A \phi(\alpha(\varepsilon_{\cC}(c_2)))\beta^{-1}(b')\\
&\overset{(\ref{tensor-prod-over-Hom-alg})}{=}&(\beta^{-1}(b)\otimes c_1)\alpha(\varepsilon_{\cC}(c_2))\otimes_A b'\\
&=&(b\otimes_A c_1\varepsilon_{\cC}(c_2))\otimes_A b'\\
&\overset{(\ref{Hom-coring-cond-2})}{=}&(b\otimes_A c)\otimes_A b'\\
&\overset{(\ref{Hom-coring-cond-2})}{=}&(b\otimes_A \varepsilon_{\cC}(c_1)c_2)\otimes_A b'\\
&\overset{(\ref{tensor-prod-over-Hom-alg})}{=}&(\beta^{-1}(b)\phi(\varepsilon_{\cC}(c_1))\otimes_A\chi(c_2))\otimes_A b'\\
&=&(\beta^{-2}(b\phi(\alpha(\varepsilon_{\cC}(c_1))))1_B\otimes_A\chi(c_2))\otimes_A\beta(\beta^{-1}(b'))\\
&=&(b\phi(\alpha(\varepsilon_{\cC}(c_1))))((1_B\otimes_A c_2)\otimes_A\beta^{-1}(b'))\\
&=&((\beta^{-1}(b)\phi(\varepsilon_{\cC}(c_1)))1_B)((1_B\otimes_A c_2)\otimes_A\beta^{-1}(b'))\\
&=&\varepsilon_{(B\mathcal{C})B}((\beta^{-1}(b)\otimes_A c_1)\otimes_A 1_B)((1_B\otimes_A c_2)\otimes_A\beta^{-1}(b')),
\end{eqnarray*}
which completes the proof that given a morphism of monoidal Hom-algebras $\phi:(A,\alpha)\to (B,\beta)$, $((B\otimes_{A}\mathcal{C})\otimes_{A}B,(\beta\otimes\chi)\otimes\beta)$ is a $(B,\beta)$-Hom-coring.
\end{proof}

\begin{example}
 A monoidal Hom-algebra $(A,\alpha)$ has a natural $(A,\alpha)$-Hom-bimodule structure with its Hom-multiplication. $(A,\alpha)$ is an $(A,\alpha)$-Hom-coring by the canonical isomorphism $A\to A\otimes_A A$, $a\mapsto \alpha^{-1}(a)\otimes 1_A$, in $\mathcal{\widetilde{H}}(\cM_k)$, as a comultiplication and the identity $A\to A$ as a counit. This Hom-coring is called a {\it trivial} $(A,\alpha)$-Hom-coring.
\end{example}

\begin{example}Let $\phi:(B,\beta)\to(A,\alpha)$ be a morphism of monoidal Hom-algebras. Then $(\cC,\chi)=(A\otimes_B A, \alpha\otimes\alpha)$ is an $(A,\alpha)$-Hom-coring with comultiplication
$$\Delta_{\cC}(a\otimes_B a')=(\alpha^{-1}(a)\otimes_B 1_A)\otimes_A (1_A\otimes_B \alpha^{-1}(a'))=(\alpha^{-1}(a)\otimes_B 1_A)\otimes_B a'$$
and counit $\varepsilon_{\cC}(a\otimes_B a')=aa'$ for all $a,a'\in A$.
\end{example}
\begin{proof} By Theorem (\ref{base-ring-extension-of-Hom-coring}), for $\phi:(B,\beta)\to(A,\alpha)$ and the trivial $(B,\beta)$-Hom-coring $(B,\beta)$ with $\Delta_B(b)=\beta^{-1}(b)\otimes_B 1_B$ and $\varepsilon_B(b)=b$, we have the base ring extension of the trivial $(B,\beta)$-Hom-coring $(B,\beta)$ to $(A,\alpha)$-Hom-coring $(AB)A=((A\otimes_B B)\otimes_B A,(\alpha\otimes\beta)\otimes\alpha)$ with

$$\Delta_{(AB)A}((a\otimes_B b)\otimes_B a')=((\alpha^{-1}(a)\otimes_B \beta^{-1}(b))\otimes_B 1_A)\otimes_A((1_A\otimes_B 1_B)\otimes_B \alpha^{-1}(a')),$$

$$\varepsilon_{(AB)A}((a\otimes_B b)\otimes_B a')=(a\phi(b))a'.$$
On the other hand we have the isomorphism $\varphi:A \to A\otimes_B B, \:\: a\mapsto \alpha^{-1}(a)\otimes_B 1_B$, in $\mathcal{\widetilde{H}}(\cM_k)$, with the inverse $\psi:A\otimes_B B\to A,\:\: a\otimes_B b \mapsto a\phi(b)$: For $a\in A$ and $b\in B$,

$$\psi(\varphi(a))=\alpha^{-1}(a)\phi(1_B)=\alpha^{-1}(a)1_A=a,$$

\begin{eqnarray*}\varphi(\psi(a\otimes_B b))&=&\varphi(a\phi(b))=\alpha^{-1}(a\phi(b))\otimes_B 1_B\\
&=&\alpha^{-1}(a)\alpha^{-1}(\phi(b))\otimes_B 1_B=\alpha^{-1}(a)\phi(\beta^{-1}(b))\otimes_B 1_B\\
&=&a\otimes_B \beta^{-1}(b)1_B=a\otimes_B b,
\end{eqnarray*}
in addition one can check that $\alpha\circ\psi=\psi\circ(\alpha\otimes\beta) $ and $(\alpha\otimes\beta)\circ\varphi=\varphi\circ\alpha$. Thus,
$(AB)A\overset{\psi\otimes 1}{\simeq}A\otimes_B A=\cC$ and

$$\Delta_{\cC}(a\otimes_B b)=((\psi\otimes id)\otimes(\psi\otimes id))\circ\Delta_{(AB)A}\circ(\varphi\otimes id)(a\otimes_B b)= (\alpha^{-1}(a)\otimes_B 1_A)\otimes_A (1_A\otimes_B \alpha^{-1}(a')),$$
$$\varepsilon_{\cC}(a\otimes_B a')=\varepsilon_{(AB)A}\circ(\varphi\otimes id)(a\otimes_B a')=aa'.$$
$(A\otimes_B A, \alpha\otimes\alpha)$ is called the {\it Sweedler} or {\it canonical} $(A,\alpha)$-{\it Hom-coring} associated to a monoidal Hom-algebra extension $\phi: (B,\beta)\to (A,\alpha)$.
\end{proof}
For the monoidal Hom-algebra $(A,\alpha)$ and the $(A,\alpha)$-Hom-coring $(\cC,\chi)$, let us put $^{*}{\cC}=\: _{A}Hom^{\mathcal{H}}(\cC,A)$, consisting of left $(A,\alpha)$-linear morphisms $f:(\cC,\chi)\to(A,\alpha)$, that is, $f(ac)=af(c)$ for $a\in A$, $c\in \cC$ and $f\circ\chi=\alpha\circ f$. Similarly, ${\cC}^{*}={Hom}_{A}^{\mathcal{H}}(\cC,A)$ and $^{*}{\cC}^{*}=\:_{A}{Hom}_{A}^{\mathcal{H}}(\cC,A)$ consist of right $(A,\alpha)$-Hom-module maps and $(A,\alpha)$-Hom-bimodule maps, respectively. Now we prove that these modules of $(A,\alpha)$-linear morphisms $\cC\to A$ have ring structures.

\begin{proposition}\label{dual-rings-of-Hom-coring}
\begin{enumerate}
\item $^{*}{\cC}$ is an associative algebra with unit $\varepsilon_{\cC}$ and multiplication
$$(f*^{l}g)(c)=f(c_1g(c_2))$$
for $f,g\in\: ^{*}{\cC}$ and $c\in \cC$.

\item ${\cC}^{*}$ is an associative algebra with unit $\varepsilon_{\cC}$ and multiplication
$$(f*^{r}g)(c)=g(f(c_1)c_2)$$
for $f,g\in {\cC}^{*}$ and $c\in \cC$.

\item $^{*}{\cC}^{*}$ is an associative algebra with unit $\varepsilon_{\cC}$ and multiplication
$$(f*g)(c)=f(c_1)g(c_2) $$
for $f,g\in\: ^{*}\cC^{*}$ and $c\in \cC$.
\end{enumerate}
\end{proposition}

\begin{proof}
\begin{enumerate}
\item For $f,g,h\in\: ^{*}{\cC}$ and $c\in \cC$,

 \begin{eqnarray*}((f*^{l}g)*^{l}h)(c)&=&f((c_1h(c_2))_1g((c_1h(c_2))_2))=f(\chi(c_{11})g(c_{12}\alpha^{-1}(h(c_2))))\\
 &=&f(\chi(c_{11})g(c_{12}h(\chi^{-1}(c_2))))\overset{(\ref{Hom-coring-cond-1})}{=}f(c_1g(c_{21}h(c_{22})))\\
 &=&(f*^{l}(g*^{l}h))(c),
 \end{eqnarray*}
where the second equality comes from the fact that $\Delta_{\cC}$ is right $(A,\alpha)$-linear, i.e., $\Delta_{\cC}(ca)=(ca)_1\otimes_A(ca)_2=\Delta_{\cC}(c)a=(c_1\otimes_A c_2)a=\chi(c_1)\otimes_A c_2\alpha^{-1}(a)$, $\forall c\in \cC$, $a\in A$.

$$(f*^{l}\varepsilon_{\cC})(c)=f(c_1\varepsilon_{\cC}(c_2))=f(c),$$
$$(\varepsilon_{\cC}*^{l}f)(c)=\varepsilon_{\cC}(c_1f(c_2))=\varepsilon_{\cC}(c_1)f(c_2)=f(\varepsilon_{\cC}(c_1)c_2)=f(c).$$
By similar computations one can prove (2) and (3).
\end{enumerate}
\end{proof}

\begin{definition} A (right-right) {\it Hom-entwining structure} is a triple $[(A,\alpha), (C,\gamma)]_{\psi}$ consisting of a monoidal Hom-algebra $(A,\alpha)$, a monoidal Hom-coalgebra $(C,\gamma)$ and a $k$-linear map $\psi: C\otimes A\to A\otimes C$ in $\widetilde{\mathcal{H}}(\mathcal{M}_k)$ satisfying the following conditions for all $a,a'\in A$, $c\in C$:
\begin{equation}\label{hom-entwining-cond-1}(aa')_{\kappa}\otimes \gamma(c)^{\kappa}=a_{\kappa}a'_{\lambda}\otimes \gamma(c^{\kappa\lambda}),\end{equation}
\begin{equation}\label{hom-entwining-cond-2}\alpha^{-1}(a_{\kappa})\otimes c^{\kappa}_{\ 1}\otimes c^{\kappa}_{\ 2}=\alpha^{-1}(a)_{\kappa\lambda}\otimes c_1^{\ \lambda}\otimes c_2^{\ \kappa},  \end{equation}
 \begin{equation}\label{hom-entwining-cond-3}1_{\kappa}\otimes c^{\kappa}=1\otimes c,\end{equation}
 \begin{equation}\label{hom-entwining-cond-4}a_{\kappa}\varepsilon(c^{\kappa})=a\varepsilon(c),\end{equation}
 where we have used the notation $\psi(c\otimes a)=a_{\kappa}\otimes c^{\kappa}$, $a\in A$, $c\in C$, for the so-called {\it entwining map} $\psi$. It is said that $(C,\gamma)$ and $(A,\alpha)$ are {\it entwined} by $\psi$. $\psi$ $\in$ $\widetilde{\mathcal{H}}(\mathcal{M}_k)$ means that the relation

 \begin{equation}\label{entwining-map-cond}\alpha(a)_{\kappa}\otimes \gamma(c)^{\kappa}=\alpha(a_{\kappa})\otimes \gamma(c^{\kappa})\end{equation} holds.

 \end{definition}

\begin{definition}\label{entwined_Hom_module}A $[(A,\alpha), (C,\gamma)]_{\psi}$-{\it entwined Hom-module} is an object $(M,\mu)$ $\in$ $\widetilde{\mathcal{H}}(\mathcal{M}_k)$ which is a right $(A,\alpha)$-Hom-module with action $\rho_M: M\otimes A\to M$, $m\otimes a\mapsto ma$ and a right $(C,\gamma)$-Hom-comodule with coaction $\rho^M:M\to M\otimes C$, $m\mapsto m_{(0)}\otimes m_{(1)}$ fulfilling the condition, for all $m \in M$, $a \in A$,

\begin{equation}\label{entwined-Hom-module-cond}\rho^{M}(ma)=m_{(0)}\alpha^{-1}(a)_{\kappa}\otimes \gamma(m_{(1)}^{\ \ \  \kappa}).\end{equation}

\end{definition}

By $\widetilde{\mathcal{M}}^C_A(\psi)$, we denote the category of $[(A,\alpha), (C,\gamma)]_{\psi}$-entwined Hom-modules together with the morphisms in which are both right $(A,\alpha)$-linear and right $(C,\gamma)$-colinear.

With the following theorem, we construct a Hom-coring associated to an entwining Hom-structure and show an identification of entwined Hom-modules with Hom-comodules of this Hom-coring, pursuing the Proposition 2.2 in \cite{Brzezinski2}.

\begin{theorem}\label{Hom-Coring-Assoc-Hom-Entw-Strc}Let $(A,\alpha)$ be a monoidal Hom-algebra and $(C,\gamma)$ be a monoidal Hom-coalgebra.
\begin{enumerate}
\item  For a Hom-entwining structure $[(A,\alpha), (C,\gamma)]_{\psi}$, $(A\otimes C,\alpha\otimes\gamma)$ is an $(A,\alpha)$-Hom-bimodule with a left Hom-module structure $a(a'\otimes c)=\alpha^{-1}(a)a'\otimes \gamma(c)$ and a right Hom-module structure $(a'\otimes c)a=a'\alpha^{-1}(a)_{\kappa}\otimes \gamma(c^{\kappa})$, for all $a, a' \in A$, $c \in C$. Furthermore, $(\mathcal{C},\chi)=(A\otimes C, \alpha\otimes\gamma)$ is an $(A,\alpha)$-Hom-coring with the comultiplication and counit
    \begin{equation}\label{comult-of-associated-Hom-coring}\Delta_{\mathcal{C}}:\mathcal{C}\to \mathcal{C}\otimes_A\mathcal{C},\ a\otimes c \mapsto (\alpha^{-1}(a)\otimes c_1)\otimes_A(1\otimes c_2),\end{equation}
    \begin{equation}\label{counit-of-associated-Hom-coring}\varepsilon_{\mathcal{C}}:\mathcal{C}\to A,\ a\otimes c \mapsto \alpha(a)\varepsilon(c).\end{equation}
\item If $\mathcal{C}=(A\otimes C,\alpha\otimes\gamma)$ is an $(A,\alpha)$-Hom-coring with the comultiplication and counit given above, then $[(A,\alpha), (C,\gamma)]_{\psi}$ is a Hom-entwining structure, where
    $$\psi:C\otimes A\to A\otimes C, \ c\otimes a\mapsto (1\otimes \gamma^{-1}(c))a.$$
\item Let $(\mathcal{C},\chi)=(A\otimes C,\alpha\otimes\gamma)$ be the $(A,\alpha)$-Hom-coring associated to $[(A,\alpha), (C,\gamma)]_{\psi}$ as in (1). Then the category of $[(A,\alpha), (C,\gamma)]_{\psi}$-entwined Hom-modules is isomorphic to the category of right $(\mathcal{C},\chi)$-Hom-comodules.
\end{enumerate}
\end{theorem}

\begin{proof}
\begin{enumerate}
\item We first show that the right Hom-action of $(A,\alpha)$ on $(A\otimes C,\alpha\otimes\gamma)$ is Hom-associative and Hom-unital, for all $a,d,e \in A$ and $c\in C$:

\begin{eqnarray*}(\alpha(a)\otimes\gamma(c))(de)&=&\alpha(a)\alpha^{-1}(de)_{\kappa}\otimes\gamma(\gamma(c)^{\kappa})\\
&=&\alpha(a)(\alpha^{-1}(d)\alpha^{-1}(e))_{\kappa}\otimes\gamma(\gamma(c)^{\kappa})\\
&\overset{(\ref{hom-entwining-cond-1})}{=}&\alpha(a)(\alpha^{-1}(d)_{\kappa}\alpha^{-1}(e)_{\lambda})\otimes\gamma^{2}(c^{\kappa\lambda})\\
&=&(a\alpha^{-1}(d)_{\kappa})\alpha(\alpha^{-1}(e)_{\lambda})\otimes\gamma(\gamma(c^{\kappa\lambda}))\\
&\overset{(\ref{entwining-map-cond})}{=}&(a\alpha^{-1}(d)_{\kappa})\alpha(\alpha^{-1}(e))_{\lambda}\otimes\gamma(\gamma(c^{\kappa})^{\lambda})\\
&=&(a\alpha^{-1}(d)_{\kappa}\otimes \gamma(c^{\kappa}))\alpha(e)\\
&=&((a\otimes c)d)\alpha(e),
\end{eqnarray*}

\begin{eqnarray*}(a\otimes c)1&=&a\alpha^{-1}(1)_{\kappa}\otimes\gamma(c^{\kappa})=a1_{\kappa}\otimes\gamma(c^{\kappa})\\
&=&\alpha^{-1}(\alpha(a))1_{\kappa}\otimes \gamma(c^{\kappa})=\alpha(a)(1_{\kappa}\otimes c^{\kappa})\\
&\overset{(\ref{hom-entwining-cond-3})}=&\alpha(a)(1\otimes c)=a1\otimes \gamma(c)\\
&=&(\alpha\otimes\gamma)(a\otimes c).
\end{eqnarray*}
One can also show that the left Hom-action, too, satisfies the Hom-associativity and Hom-unity. For any $a,b,d \in A$ and $c\in C$,

\begin{eqnarray*}(b(a\otimes c))\alpha(d)&=&(\alpha^{-1}(b)a\otimes \gamma(c))\alpha(d)=(\alpha^{-1}(b)a)\alpha^{-1}(\alpha(d))_{\kappa}\otimes\gamma(\gamma(c)^{\kappa})\\
&=&(\alpha^{-1}(b)a)\alpha(\alpha^{-1}(d))_{\kappa}\otimes \gamma(\gamma(c)^{\kappa})\overset{(\ref{entwining-map-cond})}{=}(\alpha^{-1}(b)a)\alpha(\alpha^{-1}(d)_{\kappa})\otimes \gamma^{2}(c^{\kappa})\\
&=&b(a\alpha^{-1}(d)_{\kappa})\otimes \gamma^{2}(c^{\kappa})=\alpha^{-1}(\alpha(b))(a\alpha^{-1}(d)_{\kappa})\otimes \gamma(\gamma(c^{\kappa}))\\
&=&\alpha(b)(a\alpha^{-1}(d)_{\kappa}\otimes \gamma(c^{\kappa}))=\alpha(b)((a\otimes c)d),
\end{eqnarray*}
proves the compatibility condition between left and right $(A,\alpha)$-Hom-actions.

First, it can easily be proven that the morphisms $A\otimes(\cC\otimes_A\cC)\to \cC\otimes_A\cC$,
\begin{equation}\label{left-Hom-action-on-tens-prod-of-corings} a\otimes((a'\otimes c)\otimes_A(a''\otimes c'))\mapsto\alpha^{-1}(a)(a'\otimes c)\otimes_A(\alpha(a'')\otimes \gamma(c'))\end{equation}
and $(\cC\otimes_A\cC)\otimes A\to \cC\otimes_A\cC,$
\begin{equation}\label{right-Hom-action-on-tens-prod-of-corings} ((a'\otimes c)\otimes_A(a''\otimes c'))a\mapsto(\alpha(a')\otimes \gamma(c))\otimes_A(a''\otimes c')\alpha^{-1}(a)\end{equation}
define a left Hom-action and a right Hom-action of $(A,\alpha)$ on $(\cC\otimes_A\cC,\chi\otimes\chi)$, respectively. Next it is shown that the comultiplication $\Delta_{\mathcal{C}}$ is $(A,\alpha)$-bilinear, that is, $\Delta_{\mathcal{C}}$ preserves the left and right $(A,\alpha)$-Hom-actions and the compatibility condition between them as follows: Let $a,a',b,d \in A$ and $c \in C$, then we have the following computations
\begin{eqnarray*}\Delta_{\mathcal{C}}(a(a'\otimes c))&=&(\alpha^{-1}(\alpha^{-1}(a)a')\otimes \gamma(c)_1)\otimes_A(1\otimes \gamma(c)_2)\\
&\overset{(\ref{Hom-coring-cond-2})}{=}&(\alpha^{-2}(a)\alpha^{-1}(a')\otimes \gamma(c_1))\otimes_A(1\otimes \gamma(c_2))\\
&=&\alpha^{-1}(a)(\alpha^{-1}(a')\otimes c_1)\otimes_A (\alpha(1)\otimes \gamma(c_2))\\
&\overset{(\ref{left-Hom-action-on-tens-prod-of-corings})}{=}&a((\alpha^{-1}(a')\otimes c_1)\otimes_A (1\otimes c_2))
=a\Delta_{\mathcal{C}}(a'\otimes c),
\end{eqnarray*}

\begin{eqnarray*}\Delta_{\mathcal{C}}((a'\otimes c)a)&=&\Delta_{\mathcal{C}}(a'\alpha^{-1}(a)_{\kappa}\otimes \gamma(c^{\kappa}))\\
&=&(\alpha^{-1}(a'\alpha^{-1}(a)_{\kappa})\otimes \gamma(c^{\kappa})_1)\otimes_A (1\otimes \gamma(c^{\kappa})_2)\\
&\overset{(\ref{Hom-coring-cond-2})}{=}&(\alpha^{-1}(a')\alpha^{-1}(\alpha^{-1}(a)_{\kappa})\otimes\gamma(c^{\kappa}_{\ 1}))\otimes_A(1\otimes \gamma(c^{\kappa}_{\ 2}))\\
&\overset{(\ref{hom-entwining-cond-2})}{=}&(\alpha^{-1}(a')\alpha^{-2}(a)_{\kappa\lambda}\otimes \gamma(c_1^{\ \lambda}))\otimes_A(1\otimes \gamma(c_2^{\ \kappa}))\\
&=&(\alpha^{-1}(a')\otimes c_1)\alpha(\alpha^{-2}(a)_{\kappa})\otimes_A(1\otimes \gamma(c_2^{\ \kappa}))\\
&\overset{(\ref{tensor-prod-over-Hom-alg})}{=}&(a'\otimes \gamma(c_1))\otimes_A \alpha(\alpha^{-2}(a)_{\kappa})(1\otimes c_2^{\ \kappa})\\
&=&(a'\otimes \gamma(c_1))\otimes_A (\alpha(\alpha^{-2}(a)_{\kappa})\otimes \gamma(c_2^{\ \kappa}))\\
&=&(a'\otimes \gamma(c_1))\otimes_A (1\alpha^{-1}(\alpha^{-1}(a))_{\kappa}\otimes \gamma(c_2^{\ \kappa}))\\
&=&(a'\otimes \gamma(c_1))\otimes_A (1\otimes c_2)\alpha^{-1}(a)\\
&\overset{(\ref{right-Hom-action-on-tens-prod-of-corings})}{=}&((\alpha^{-1}(a')\otimes c_1)\otimes_A(1\otimes c_2))a
=\Delta_{\mathcal{C}}(a'\otimes c)a,
\end{eqnarray*}

\begin{eqnarray*}\alpha(b)(\Delta_{\mathcal{C}}(a\otimes c)d)&=&\alpha(b)(((\alpha^{-1}(a)\otimes c_1)\otimes_A(1 \otimes c_2))d)\\
&=&\alpha(b)((a\otimes \gamma(c_1))\otimes_A(1\otimes c_2)\alpha^{-1}(d))\\
&\overset{(\ref{right-Hom-action-on-tens-prod-of-corings})}{=}&\alpha(b)((a\otimes \gamma(c_1))\otimes_A (1\alpha^{-1}(\alpha^{-1}(d))_{\kappa}\otimes \gamma(c_2^{\ \kappa})))\\
&=&\alpha(b)((a\otimes \gamma(c_1))\otimes_A (\alpha(\alpha^{-2}(d)_{\kappa})\otimes \gamma(c_2^{\ \kappa})))\\
&\overset{(\ref{left-Hom-action-on-tens-prod-of-corings})}{=}&b(a\otimes \gamma(c_1))\otimes_A (\alpha^{2}(\alpha^{-2}(d)_{\kappa})\otimes \gamma^{2}(c_2^{\ \kappa}))\\
&=&(\alpha^{-1}(b)a\otimes \gamma^{2}(c_1))\otimes_A \alpha^{2}(\alpha^{-2}(d)_{\kappa})(1\otimes \gamma(c_2^{\ \kappa}))\\
&\overset{(\ref{tensor-prod-over-Hom-alg})}{=}&(\alpha^{-1}(\alpha^{-1}(b)a)\otimes \gamma(c_1))\alpha^{2}(\alpha^{-2}(d)_{\kappa}) \otimes_A(1\otimes \gamma^{2}(c_2^{\ \kappa}))\\
&=&((\alpha^{-2}(b)\alpha^{-1}(a))\alpha(\alpha^{-2}(d)_{\kappa})_{\lambda}\otimes \gamma(\gamma(c_1)^{\lambda}))\otimes_A(1\otimes \gamma^{2}(c_2^{\ \kappa}))\\
&\overset{(\ref{entwining-map-cond})}{=}&((\alpha^{-2}(b)\alpha^{-1}(a))\alpha(\alpha^{-2}(d)_{\kappa\lambda})\otimes \gamma(\gamma(c_1^{\ \lambda})))\otimes_A(1\otimes \gamma^{2}(c_2^{\ \kappa}))\\
&=&(\alpha^{-1}(b)(\alpha^{-1}(a)\alpha^{-2}(d)_{\kappa\lambda})\otimes \gamma^{2}(c_1^{\ \lambda}))\otimes_A(1\otimes \gamma^{2}(c_2^{\ \kappa}))\\
&=&(\alpha^{-1}(b)(\alpha^{-1}(a)\alpha^{-1}(\alpha^{-1}(d))_{\kappa\lambda})\otimes \gamma^{2}(c_1^{\ \lambda}))\otimes_A(1\otimes \gamma^{2}(c_2^{\ \kappa}))\\
&\overset{(\ref{hom-entwining-cond-2})}{=}&(\alpha^{-1}(b)(\alpha^{-1}(a)\alpha^{-1}(\alpha^{-1}(d)_{\kappa}))\otimes \gamma^{2}(c_{\ 1}^{\kappa}))\otimes_A (1\otimes \gamma^{2}(c_{\ 2}^{\kappa}))\\
&=&((\alpha^{-2}(b)\alpha^{-1}(a))\alpha^{-1}(d)_{\kappa}\otimes\gamma^{2}(c_{\ 1}^{\kappa}))\otimes_A (1\otimes \gamma^{2}(c_{\ 2}^{\kappa}))\\
&\overset{(\ref{Hom-coring-cond-2})}{=}&(\alpha^{-1}((\alpha^{-1}(b)a)\alpha(\alpha^{-1}(d)_{\kappa}))\otimes \gamma^{2}(c^{\kappa})_1)\otimes_A (1\otimes \gamma^{2}(c^{\kappa})_2)\\
&=&\Delta_{\mathcal{C}}((\alpha^{-1}(b)a)\alpha(\alpha^{-1}(d)_{\kappa})\otimes \gamma^{2}(c^{\kappa}))
=\Delta_{\mathcal{C}}(b(a\otimes c))\alpha(d).
\end{eqnarray*}

One easily checks that the counit $\varepsilon_{\mathcal{C}}$ is both left and right $(A,\alpha)$-linear. For any $a,b,d \in A$ and $c \in C$ we have
\begin{eqnarray*}\varepsilon_{\mathcal{C}}((b(a\otimes c))\alpha(d))&=&\varepsilon_{\mathcal{C}}(b(a\alpha^{-1}(d)_{\kappa})\otimes \gamma^{2}(c^{\kappa}))\\
&=&\alpha(b(a\alpha^{-1}(d)_{\kappa}))\varepsilon(\gamma^{2}(c^{\kappa}))\\
&\overset{(\ref{Hom-coring-cond-2})}{=}&\alpha(b)(\alpha(a)\alpha(\alpha^{-1}(d)_{\kappa}))\varepsilon(c^{\kappa})\\
&=&\alpha(b)(\alpha(a)\alpha(\alpha^{-1}(d)_{\kappa}\varepsilon(c^{\kappa})))\\
&\overset{(\ref{hom-entwining-cond-4})}{=}&\alpha(b)(\alpha(a)\alpha(\alpha^{-1}(d)\varepsilon(c)))\\
&=&\alpha(b)(\alpha(a)\varepsilon(c)d)=\alpha(b)(\varepsilon_{\mathcal{C}}(a\otimes c)d).
\end{eqnarray*}
This finishes the proof that $\varepsilon_{\mathcal{C}}$ is $(A,\alpha)$-bilinear.
Let us put
$$\Delta_{\mathcal{C}}(a\otimes c)=(a\otimes c)_1\otimes_A(a\otimes c)_2=(\alpha^{-1}(a)\otimes c_1)\otimes_A(1\otimes c_2).$$

Then we get the following
\begin{eqnarray*}\lefteqn{(\alpha^{-1}\otimes\gamma^{-1})((a\otimes c)_1)\otimes_A\Delta_{\mathcal{C}}((a\otimes c)_2)}\hspace{8em}\\
&=&(\alpha^{-2}(a)\otimes\gamma^{-1}(c_1))\otimes_A((1\otimes c_{21})\otimes_A(1\otimes c_{22}))\\
&=&(\alpha^{-2}(a)\otimes c_{11})\otimes_A((1\otimes c_{12})\otimes_A(1\otimes \gamma^{-1}(c_2)))\\
&=&(\alpha^{-1}(a)\otimes c_{1})_1\otimes_A((\alpha^{-1}(a)\otimes c_{1})_2\otimes_A(1\otimes \gamma^{-1}(c_2)))\\
&=&(a\otimes c)_{11}\otimes_A((a\otimes c)_{12}\otimes_A(\alpha^{-1}\otimes\gamma^{-1})((a\otimes c)_2)),
\end{eqnarray*}
where in the second step the Hom-coassociativity of $(C,\gamma)$ is used.
\begin{eqnarray*}\varepsilon_{\mathcal{C}}((a\otimes c)_1)(a\otimes c)_2&=&\varepsilon_{\mathcal{C}}((\alpha^{-1}(a)\otimes c_1))(1\otimes c_2)\\
&=&\alpha(\alpha^{-1}(a)\varepsilon(c_1))(1\otimes c_2)=a(1\otimes \varepsilon(c_1)c_2)\\
&=&a(1\otimes \gamma^{-1}(c))=a\otimes c,
\end{eqnarray*}
on the other hand we have
\begin{eqnarray*}(a\otimes c)_1\varepsilon_{\mathcal{C}}((a\otimes c)_2)&=&(\alpha^{-1}(a)\otimes c_1)\alpha(1)\varepsilon(c_2)\\
&=&(\alpha^{-1}(a)\otimes c_1\varepsilon(c_2))1=(\alpha^{-1}(a)\otimes \gamma^{-1}(c))1=a\otimes c.
\end{eqnarray*}

We also show that the following relations
\begin{eqnarray*}\Delta_{\mathcal{C}}(\alpha(a)\otimes\gamma(c))&=&(\alpha^{-1}(\alpha(a))\otimes \gamma(c)_1)\otimes_A(1\otimes \gamma(c)_2)\\
&=&(\alpha(\alpha^{-1}(a))\otimes \gamma(c_1))\otimes_A(\alpha(1)\otimes \gamma(c_2)\\
&=&((\alpha\otimes\gamma)\otimes(\alpha\otimes\gamma))(\Delta_{\mathcal{C}}(a\otimes c)),
\end{eqnarray*}
\begin{equation*}\varepsilon_{\mathcal{C}}(\alpha(a)\otimes\gamma(c))=\alpha(\alpha(a))\varepsilon(\gamma(c))
=\alpha(\alpha(a))\varepsilon(c)
=\alpha(\varepsilon_{\mathcal{C}}(a\otimes c))
\end{equation*}
hold, which completes the proof that $(A\otimes C, \alpha\otimes\gamma)$ is an $(A,\alpha)$-Hom-coring.

\item Let us denote $\psi(c\otimes a)=(1\otimes \gamma^{-1}(c))a=a_{\kappa}\otimes c^{\kappa}$. $\psi$ is in $\widetilde{\mathcal{H}}(\mathcal{M}_k)$:
\begin{eqnarray*}(\alpha\otimes\gamma)(\psi(c\otimes a))&=&\alpha(a_{\kappa})\otimes\gamma(c^{\kappa})=(\alpha\otimes\gamma)((1\otimes \gamma^{-1}(c))a)\\
&=&(\alpha(1)\otimes\gamma(\gamma^{-1}(c)))\alpha(a)=(1\otimes c)\alpha(a)\\
&=&(1\otimes \gamma^{-1}(\gamma(c)))\alpha(a)=\alpha(a)_{\kappa}\otimes\gamma(c)^{\kappa}\\
&=&\psi(\gamma(c)\otimes\alpha(a)),
\end{eqnarray*}
where in the third equality the fact that the right Hom-action of  $(A,\alpha)$ on $(A\otimes C,\alpha\otimes\gamma)$ is a morphism in $\widetilde{\mathcal{H}}(\mathcal{M}_k)$ was used. Now, let $a,a'\in A$ and $c\in C$, then
\begin{eqnarray*}\psi(c\otimes aa')&=&(aa')_{\kappa}\otimes c^{\kappa}=(1\otimes \gamma^{-1}(c))(aa')\\
&=&((\alpha^{-1}(1)\otimes \gamma^{-1}(\gamma^{-1}(c)))a)\alpha(a')=((1\otimes \gamma^{-1}(\gamma^{-1}(c)))a)\alpha(a')\\
&=&(a_{\kappa}\otimes\gamma^{-1}(c)^{\kappa})\alpha(a')=(\alpha^{-1}(a_{\kappa})1\otimes \gamma(\gamma^{-1}(\gamma^{-1}(c)^{\kappa})))\alpha(a')\\
&=&(a_{\kappa}(1\otimes\gamma^{-1}(\gamma^{-1}(c)^{\kappa})))\alpha(a')\\
&=&\alpha(a_{\kappa})((1\otimes\gamma^{-1}(\gamma^{-1}(c)^{\kappa}))a')=\alpha(a_{\kappa})\psi(\gamma^{-1}(c)^{\kappa}\otimes a')\\
&=&\alpha(a_{\kappa})(a'_{\lambda}\otimes \gamma^{-1}(c)^{\kappa\lambda})=\alpha^{-1}(\alpha(a_{\kappa}))a'_{\lambda}\otimes \gamma(\gamma^{-1}(c)^{\kappa\lambda})\\
&=&a_{\kappa}a'_{\lambda}\otimes \gamma(\gamma^{-1}(c)^{\kappa\lambda}).
\end{eqnarray*}
In the above equality, if we replace $c$ by $\gamma(c)$ we obtain $(aa')_{\kappa}\otimes \gamma(c)^{\kappa}=a_{\kappa}a'_{\lambda}\otimes \gamma(c^{\kappa\lambda})$. Next, by using the right $(A,\alpha)$-linearity of $\Delta_{\mathcal{C}}$ we prove the following
\begin{eqnarray*}\alpha^{-1}(a)_{\kappa\lambda}\otimes c_1^{\ \lambda}\otimes c_2^{\ \kappa}&=&\psi(c_1\otimes \alpha^{-1}(a)_{\kappa})\otimes c_2^{\ \kappa}\\
&=&(1\otimes \gamma^{-1}(c_1))\alpha^{-1}(a)_{\kappa}\otimes c_2^{\ \kappa}\\
&=&(1\otimes \gamma^{-1}(c_1))\alpha^{-1}(a)_{\kappa}\otimes_A (1\otimes \gamma^{-1}(c_2^{\ \kappa}))\\
&\overset{(\ref{tensor-prod-over-Hom-alg})}{=}&(1\otimes c_1)\otimes_A \alpha^{-1}(a)_{\kappa}(1\otimes \gamma^{-2}(c_2^{\ \kappa}))\\
&=&(1\otimes c_1)\otimes_A (\alpha^{-1}(a)_{\kappa}\otimes \gamma^{-1}(c_2^{\ \kappa}))\\
&=&(id_{A\otimes C}\otimes id_A\otimes\gamma^{-1})((1\otimes c_1)\otimes_A \psi(c_2\otimes\alpha^{-1}(a)))\\
&=&(id_{A\otimes C}\otimes id_A\otimes\gamma^{-1})((1\otimes c_1)\otimes_A((1\otimes\gamma^{-1}(c_2))\alpha^{-1}(a)))\\
&\overset{(\ref{right-Hom-action-on-tens-prod-of-corings})}{=}&(id_{A\otimes C}\otimes id_A\otimes\gamma^{-1})(((\alpha^{-1}(1)\otimes \gamma^{-1}(c_1))\otimes_A(1\otimes\gamma^{-1}(c_2)))a)\\
&\overset{(\ref{Hom-coring-cond-2})}{=}&(id_{A\otimes C}\otimes id_A\otimes\gamma^{-1})(((1\otimes \gamma^{-1}(c)_1)\otimes_A(1\otimes\gamma^{-1}(c)_2))a)\\
&=&(id_{A\otimes C}\otimes id_A\otimes\gamma^{-1})(\Delta_{\mathcal{C}}(1\otimes \gamma^{-1}(c))a)\\
&=&(id_{A\otimes C}\otimes id_A\otimes\gamma^{-1})(\Delta_{\mathcal{C}}((1\otimes \gamma^{-1}(c))a))\\
&=&(id_{A\otimes C}\otimes id_A\otimes\gamma^{-1})(\Delta_{\mathcal{C}}(a_{\kappa}\otimes c^{\kappa}))\\
&=&(id_{A\otimes C}\otimes id_A\otimes\gamma^{-1})((\alpha^{-1}(a_{\kappa})\otimes c^{\kappa}_{\ 1})\otimes_A (1\otimes c^{\kappa}_{\ 2}))\\
&=&(\alpha^{-1}(a_{\kappa})\otimes c^{\kappa}_{\ 1})\otimes_A (1\otimes \gamma^{-1}(c^{\kappa}_{\ 2}))\\
&=&(\alpha^{-1}(\alpha^{-1}(a_{\kappa}))\otimes\gamma^{-1}(c^{\kappa}_{\ 1}))1\otimes \gamma(\gamma^{-1}(c^{\kappa}_{\ 2}))
=\alpha^{-1}(a_{\kappa})\otimes c_{\ 1}^{\kappa}\otimes c^{\kappa}_{\ 2}.
\end{eqnarray*}
We also find $\psi(c\otimes 1)=1_{\kappa}\otimes c^{\kappa}=(1\otimes \gamma^{-1}(c))1=1\otimes c.$
Finally, the fact of $\varepsilon_{\mathcal{C}}$ being right $(A,\alpha)$-linear gives
\begin{eqnarray*}\alpha(a_{\kappa})\varepsilon(c^{\kappa})&=&\varepsilon_{\mathcal{C}}(a_{\kappa}\otimes c^{\kappa})=\varepsilon_{\mathcal{C}}((1\otimes \gamma^{-1}(c))a)\\
&=&\varepsilon_{\mathcal{C}}(1\otimes \gamma^{-1}(c))a=\alpha(1)\varepsilon(\gamma^{-1}(c))a=1a\varepsilon(c)=\alpha(a)\varepsilon(c),
\end{eqnarray*}
which means that $a_{\kappa}\varepsilon(c^{\kappa})=a\varepsilon(c)$. Therefore $[(A,\alpha), (C,\gamma)]_{\psi}$ is a Hom-entwining structure.
\item The essential point is that if $(M,\mu)$ is a right $(A,\alpha)$-Hom-module, then $(M\otimes C,\mu\otimes\gamma)$ is a right $(A,\alpha)$-Hom-module with the Hom-action $\rho_{M\otimes C}:(M\otimes C)\otimes A\to M\otimes C$, $(m\otimes c)\otimes a \mapsto (m\otimes c)a=m\alpha^{-1}(a)_{\kappa}\otimes \gamma(c^{\kappa})$. $\rho_{M\otimes C}$ indeed satisfies Hom-associativity and Hom-unity as follows. For all $m\in M$, $a,a'\in A$ and $c\in C$,
\begin{eqnarray*}(\mu(m)\otimes\gamma(c))(aa')&=&\mu(m)\alpha^{-1}(aa')_{\kappa}\otimes \gamma(\gamma(c)^{\kappa})\\
&\overset{(\ref{hom-entwining-cond-1})}{=}&\mu(m)(\alpha^{-1}(a)_{\kappa}\alpha^{-1}(a')_{\lambda})\otimes \gamma(\gamma(c^{\kappa\lambda}))\\
&=&(m\alpha^{-1}(a)_{\kappa})\alpha(\alpha^{-1}(a')_{\lambda})\otimes \gamma(\gamma(c^{\kappa\lambda}))\\
&\overset{(\ref{entwining-map-cond})}{=}&(m\alpha^{-1}(a)_{\kappa})\alpha(\alpha^{-1}(a'))_{\lambda}\otimes \gamma(\gamma(c^{\kappa})^{\lambda})\\
&=&(m\alpha^{-1}(a)_{\kappa})\alpha^{-1}(\alpha(a'))_{\lambda}\otimes \gamma(\gamma(c^{\kappa})^{\lambda})\\
&=&(m\alpha^{-1}(a)_{\kappa}\otimes \gamma(c^{\kappa}))\alpha(a')=((m\otimes c)a)\alpha(a'),
\end{eqnarray*}
\begin{eqnarray*}(m\otimes c)1&=&m\alpha^{-1}(1)_{\kappa}\otimes \gamma(c^{\kappa})=m1_{\kappa}\otimes \gamma(c^{\kappa})
\overset{(\ref{hom-entwining-cond-3})}{=}m1\otimes \gamma(c)=\mu(m)\otimes \gamma(c).
\end{eqnarray*}
With respect to this Hom-action of $(A,\alpha)$ on $(M\otimes C,\mu\otimes\gamma)$, becoming an $[(A,\alpha), (C,\gamma)]_{\psi}$-entwined Hom-module is equivalent to the fact that the Hom-coaction of $(C,\gamma)$ on $(M,\mu)$ is right $(A,\alpha)$-linear.
Let $(M,\mu)\in \widetilde{\mathcal{M}}^{C}_A(\psi)$ with the right $(C,\gamma)$-Hom-comodule structure $m\mapsto m_{(0)}\otimes m_{(1)}$. Then $(M,\mu) \in \widetilde{\mathcal{M}}^{\mathcal{C}}$ with the Hom-coaction $\rho^{M}:M \to M\otimes_A \mathcal{C}$, $m\mapsto m_{(0)}\otimes_A(1\otimes \gamma^{-1}(m_{(1)}))$, which actually is
\begin{eqnarray*}\rho^{M}(m)&=&m_{(0)}\otimes_A(1\otimes \gamma^{-1}(m_{(1)}))=\mu^{-1}(m)1\otimes \gamma(\gamma^{-1}(m_{(1)}))\\
&=&m_{(0)}\otimes m_{(1)},
\end{eqnarray*}

where in the second equality we have used the canonical identification $$\phi:M\otimes_A(A\otimes C) \simeq M\otimes C,\ m\otimes_A(a\otimes c)\mapsto \mu^{-1}(m)a\otimes \gamma(c),$$ and $\rho^{M}$ is $(A,\alpha)$-linear since $$\rho^{M}(ma)=(ma)_{(0)}\otimes (ma)_{(1)}=m_{(0)}\alpha^{-1}(a)_{\kappa}\otimes\gamma(m_{(1)}^{\ \ \ \kappa})=(m_{(0)}\otimes m_{(1)})a.$$
Conversely, if $(M,\mu)$ is a right $(A\otimes C, \alpha\otimes\gamma)$-Hom-comodule with the coaction $\rho^{M}:M\to M\otimes_A(A\otimes C)$, by using the canonical identification above, one gets the $(C,\gamma)$-Hom-comodule structure $\bar{\rho}^{M}=\phi\circ\rho^{M}:M\to M\otimes C$ on $(M,\mu)$. One can also check that $\phi$ is right $(A,\alpha)$-linear once the following $(A,\alpha)$-Hom-module structure on $M\otimes_A\mathcal{C}$ is given:

$$\rho_{M\otimes_A \mathcal{C}}:(M\otimes_A\mathcal{C})\otimes A\to M\otimes_A\mathcal{C},\  (m\otimes_A(a\otimes c))\otimes a'\mapsto \mu(m)\otimes_A(a\otimes c)\alpha^{-1}(a'),$$

thus $\bar{\rho}^{M}$ is $(A,\alpha)$-linear since by definition $\rho^{M}$ is $(A,\alpha)$-linear. Therefore $(M,\mu)$ has an $[(A,\alpha), (C,\gamma)]_{\psi}$-entwined Hom-module structure.
\end{enumerate}
\end{proof}
One should refer to both \cite[Proposition 25]{CaenepeelMilitaruZhu} and \cite[Item 32.9]{BrzezinskiWisbauer} for the classical version of the following theorem.
\begin{theorem}\label{dual-algebra-of-Hom-coring-to-entw}Let $[(A,\alpha), (C,\gamma)]_{\psi}$be an entwining Hom-structure and $(\cC,\chi)=(A\otimes C,\alpha\otimes\gamma)$ be the associated $(A,\alpha)$-Hom-coring. Then the so-called {\it Koppinen smash} or $\psi$-twisted convolution algebra $Hom_{\psi}^{\mathcal{H}}(C,A)=(Hom^{\mathcal{H}}(C,A),*_{\psi},\eta_A\circ\varepsilon_C)$, where $(f*_{\psi}g)(c)=f(c_2)_{\kappa}g(c_1^{\ \kappa})$ for any $f,g \in Hom^{\mathcal{H}}(C,A)$, is anti-isomorphic to the algebra $(^{*}\cC,*^{l},\varepsilon_{\cC})$ in Proposition (\ref{dual-rings-of-Hom-coring}).
\end{theorem}

\begin{proof}For $f,g,h \in Hom^{\mathcal{H}}(C,A)$ and $c\in C$,

\begin{eqnarray*}\lefteqn{((f*_{\psi}g)*_{\psi} h)(c)}\hspace{2em}\\
&=&(f*_{\psi} g)(c_2)_{\kappa}h(c_1^{\ \kappa})=(f(c_{22})_{\lambda}g(c_{21}^{\ \ \lambda}))h(c_1^{\ \kappa})\\
&\overset{(\ref{hom-entwining-cond-1})}{=}&(f(c_{22})_{\lambda\kappa}g(c_{21}^{\ \ \lambda})_{\sigma})h(\gamma(\gamma^{-1}(c_1)^{\kappa\sigma}))=(f(c_{22})_{\lambda\kappa}g(c_{21}^{\ \ \lambda})_{\sigma})
\alpha(h(\gamma^{-1}(c_1)^{\kappa\sigma}))\\
&=&\alpha(f(c_{22})_{\lambda\kappa})(g(c_{21}^{\ \ \lambda})_{\sigma}h(\gamma^{-1}(c_1)^{\kappa\sigma}))\overset{\kappa\leftrightarrow\lambda}{=}\alpha(f(c_{22})_{\kappa\lambda})(g(c_{21}^{\ \ \kappa})_{\sigma}h(\gamma^{-1}(c_1)^{\lambda\sigma}))\\
&\overset{(\ref{Hom-coring-cond-1})}{=}&\alpha(f(\gamma^{-1}(c_{2}))_{\kappa\lambda})(g(c_{12}^{\ \ \kappa})_{\sigma}h(c_{11}^{\ \ \lambda\sigma}))
=\alpha(\alpha^{-1}(f(c_{2}))_{\kappa\lambda})(g(c_{12}^{\ \ \kappa})_{\sigma}h(c_{11}^{\ \ \lambda\sigma}))\\
&\overset{(\ref{hom-entwining-cond-2})}{=}&f(c_2)_{\kappa}(g(c_{1\ 2}^{\ \kappa})_{\sigma}h(c_{1\ 1}^{\ \kappa\ \sigma}))\\
&=&f(c_2)_{\kappa}(g*_{\psi} h)(c_1^{\ \kappa})\\
&=&(f*_{\psi}(g*_{\psi} h))(c),
\end{eqnarray*}
proving that $*_{\psi}$ is associative. Now we show that $\eta\varepsilon$ is the unit for $*_{\psi}$:

\begin{eqnarray*}(\eta\varepsilon*_{\psi} f)(c)&=&\eta\varepsilon(c_2)_{\kappa}f(c_1^{\ \kappa})=\varepsilon(c_2)1_{\kappa}f(c_1^{\ \kappa})\\
&=&1_{\kappa}f(\gamma^{-1}(c)^{\kappa})\overset{(\ref{hom-entwining-cond-3})}{=}1f(\gamma^{-1}(c))\\
&=&f(c)\\
&=&f(\gamma^{-1}(c))1=f(c_2)\varepsilon(c_1)1\\
&\overset{(\ref{hom-entwining-cond-4})}{=}&f(c_2)_{\kappa}\varepsilon(c_1^{\ \kappa})1=f(c_2)_{\kappa}\eta\varepsilon(c_1^{\ \kappa})\\
&=&(f*_{\psi }\eta\varepsilon)(c).
\end{eqnarray*}
The map $\phi:\: ^{*}{\cC}=\: _{A}Hom^{\mathcal{H}}(A\otimes C,A)\to Hom^{\mathcal{H}}(C,A)$ given by

\begin{equation}\phi(\xi)(c)= \xi(1\otimes \gamma^{-1}(c))\end{equation}

for any $\xi\in \: ^{*}{\cC} $ and $c\in C$, is a $k$-module isomorphism with the inverse $\varphi:Hom^{\mathcal{H}}(C,A) \to \: ^{*}{\cC} $ given by $\varphi(f)(a\otimes c)=af(c)$ for all $f\in Hom^{\mathcal{H}}(C,A)$ and $a\otimes c \in A\otimes C$: Let $a\in A$, $a'\otimes c \in A\otimes C$ and $f\in Hom^{\mathcal{H}}(C,A)$. Then

\begin{eqnarray*}\varphi(f)(a(a'\otimes c))&=&\varphi(f)(\alpha^{-1}(a)a'\otimes \gamma(c))=(\alpha^{-1}(a)a')f(\gamma(c))\\
&=&(\alpha^{-1}(a)a')\alpha(f(c))=a(a'f(c))=a\varphi(f)(a'\otimes c)
\end{eqnarray*}
and $$\varphi(f)(\alpha(a)\otimes \gamma(c))=\alpha(a)f(\gamma(c))=\alpha(af(c))=\alpha(\varphi(f)(a\otimes c)),$$ showing that $\varphi(f)$ is $(A,\alpha)$-linear. On the other hand,
$$\varphi(\phi(\xi))(a\otimes c)=a\phi(\xi)(c)=a\xi(1\otimes\gamma^{-1}(c))=\xi(a(1\otimes\gamma^{-1}(c)))=\xi(a\otimes c),$$
$$\phi(\varphi(f))(c)=\varphi(f)(1\otimes \gamma^{-1}(c))=1f(\gamma^{-1}(c))=f(c).$$
Now if we put $\phi(\xi)=f$ and $\phi(\xi')=f'$, we have $f(c)=\xi(1\otimes\gamma^{-1}(c))$, $f'(c)=\xi'(1\otimes\gamma^{-1}(c))$ for $c\in C$, and then

\begin{eqnarray*}(\xi*^{l} \xi')(a\otimes c)&=&\xi((a\otimes c)_1\xi'((a\otimes c)_2))\\
&\overset{(\ref{comult-of-associated-Hom-coring})}{=}&\xi((\alpha^{-1}(a)\otimes c_1)\xi'(1\otimes c_2))\\
&=&\xi((\alpha^{-1}(a)\otimes c_1)f'(\gamma(c_2)))=\xi((\alpha^{-1}(a)\otimes c_1)\alpha(f'(c_2))) \\
&=&\xi(\alpha^{-1}(a)\alpha^{-1}(\alpha(f'(c_2)))_{\kappa}\otimes \gamma(c_1^{\ \kappa}))=\xi(\alpha^{-1}(a)f'(c_2)_{\kappa}\otimes \gamma(c_1^{\ \kappa}))\\
&=&(\alpha^{-1}(a)f'(c_2)_{\kappa})f(\gamma(c_1^{\ \kappa}))=(\alpha^{-1}(a)f'(c_2)_{\kappa})\alpha(f(c_1^{\ \kappa}))\\
&=&a(f'(c_2)_{\kappa}f(c_1^{\ \kappa}))=a(f'*_{\psi f})(c),
\end{eqnarray*}
which induces the following

\begin{eqnarray*}\phi(\xi*^{l} \xi')(c)&=&(\xi*^{l} \xi')(1\otimes\gamma^{-1}(c))=1(f'*_{\psi} f)(\gamma^{-1}(c))\\
&=&\alpha((f'*_{\psi} f)(\gamma^{-1}(c)))=(f'*_{\psi} f)(\gamma(\gamma^{-1}(c)))\\
&=&(f'*_{\psi} f)(c)=(\phi(\xi')*_{\psi} \phi(\xi))(c).
\end{eqnarray*}
Moreover, $\phi(\varepsilon_{\cC})(c)=\varepsilon_{\cC}(1\otimes\gamma^{-1}(c))=\alpha(1)\varepsilon(\gamma^{-1}(c))=\eta\varepsilon(c).$
Therefore $\phi$ is the anti-isomorphism of the algebras $^{*}\cC$ and $Hom_{\psi}^{\mathcal{H}}(C,A)$.
\end{proof}
\section{Entwinings and Hom-Hopf-type Modules}
A bialgebra in $\widetilde{\mathcal{H}}(\mathcal{M}_k)$ is called a monoidal Hom-bialgebra (see \cite{CaenepeelGoyvaerts}), i.e. a monoidal Hom-bialgebra $(H,\alpha)$ is a sextuple $(H,\alpha,m,\eta,\Delta,\varepsilon)$ where $(H,\alpha,m,\eta)$ is a monoidal Hom-algebra and $(H,\alpha,\Delta,\varepsilon)$ is a monoidal Hom-coalgebra  such that
    \begin{equation}\label{monoidal-Hom-bialg-cond-1}\Delta(hh')=\Delta(h)\Delta(h')\:;\: \Delta(1_H)=1_H\otimes1_H,\end{equation}
    \begin{equation}\label{monoidal-Hom-bialg-cond-2}\varepsilon(hh')=\varepsilon(h)\varepsilon(h')\:;\: \varepsilon(1_H)=1,\end{equation}
    for any $h,h' \in H$.

\begin{definition}\cite{ChenZhang1} Let $(B,\beta)$ be a monoidal Hom-bialgebra. A {\it right} $(B,\beta)$-{\it Hom-comodule algebra} $(A,\alpha)$ is a monoidal Hom-algebra and a right $(B,\beta)$-Hom-comodule with a Hom-coaction $\rho^{A}:A\to A\otimes B,\: a\mapsto a_{(0)}\otimes a_{(1)}$ such that $\rho^{A}$ is a Hom-algebra morphism, i.e., for any $a,a' \in A$
\begin{equation}\label{Hom-comodule-algebra-cond}(aa')_{(0)}\otimes (aa')_{(1)}=a_{(0)}a'_{(0)}\otimes a_{(1)}{a'}_{(1)},\:\: \rho^{A}(1_A)=1_A\otimes 1_B,\end{equation}
$$\rho^{A}\circ \alpha=(\alpha\otimes\beta)\circ\rho^{A}.$$
\end{definition}

\begin{definition}Let $(B,\beta)$ be a monoidal Hom-bialgebra. A {\it right} $(B,\beta)$-{\it Hom-module coalgebra} $(C,\gamma)$ is a monoidal Hom-coalgebra and a right $(B,\beta)$-Hom-module with the Hom-action $\rho_{C}:C\otimes B\to C,\: c\otimes b \mapsto cb$ such that $\rho_{C}$ is a Hom-coalgebra morphism, that is, for any $c\in C$ and $b\in B$
\begin{equation}\label{Hom-module-coalgebra-cond}(cb)_1\otimes (cb)_2=c_1b_1\otimes c_2b_2,\: \varepsilon_{C}(cb)=\varepsilon_{C}(c)\varepsilon_{B}(b),\end{equation}
$$\rho_{C}\circ(\gamma\otimes\beta)=\gamma\circ\rho_{C}.$$
\end{definition}

By the following construction, we show that a Hom-Doi-Koppinen datum comes from a Hom-entwining structure and that the Doi-Koppinen Hom-Hopf modules are the same as the associated entwined Hom-modules, and give the structure of Hom-coring corresponding to the relevant Hom-entwining structure.

\begin{proposition}\label{Hom-coring-assoc-Doi-Koppinen-datum} Let $(B,\beta)$ be a monoidal Hom-bialgebra. Let $(A,\alpha)$ be a right $(B,\beta)$-Hom-comodule algebra with Hom-coaction $\rho^{A}:A\to A\otimes B,\: a\mapsto a_{(0)}\otimes a_{(1)}$ and $(C,\gamma)$ be a right $(B,\beta)$-Hom-module coalgebra with Hom-action $\rho_{C}:C\otimes B \to C,\: c\otimes b\mapsto cb$. Define the morphism
\begin{equation}\label{Doi-Koppinen-entwining}\psi:C\otimes A\to A\otimes C,\: c\otimes a\mapsto \alpha(a_{(0)})\otimes \gamma^{-1}(c)a_{(1)}=a_{\kappa}\otimes c^{\kappa}. \end{equation}
Then the following assertions hold.
\begin{enumerate}
\item $[(A,\alpha), (C,\gamma)]_{\psi}$ is an Hom-entwining structure.

\item $(M,\mu)$ is an $[(A,\alpha), (C,\gamma)]_{\psi}$-entwined Hom-module if and only if it is a right $(A,\alpha)$-Hom-module with $\rho_{M}:M\otimes A\to M,\: m\otimes a\mapsto ma$ and a right $(C,\gamma)$-Hom-comodule with $\rho^{M}:M\to M\otimes C,\: m\mapsto m_{(0)}\otimes m_{(1)}$ such that

\begin{equation}\label{Doi-Koppinen-Hom-module-cond}\rho^{M}(ma)=m_{(0)}a_{(0)}\otimes m_{(1)}a_{(1)} \end{equation}

for any $m\in M$ and $a\in A$.

\item $(\cC,\chi)=(A\otimes C, \alpha\otimes\gamma)$ is an $(A,\alpha)$-Hom-coring with comultiplication and counit given by (\ref{comult-of-associated-Hom-coring}) and (\ref{counit-of-associated-Hom-coring}), respectively, and it has the $(A,\alpha)$-Hom-bimodule structure $a(a'\otimes c)=\alpha^{-1}(a)a'\otimes \gamma(c)$, $(a'\otimes c)a=a'a_{(0)}\otimes ca_{(1)}$ for $a,a'\in A$ and $c\in C$.

\item $Hom^{\mathcal{H}}(C,A)$ is an associative algebra with the unit $\eta\varepsilon$ and the multiplication $*_{\psi}$ defined by

     \begin{equation}(f*_{\psi} g)(c)= \alpha(f(c_2)_{(0)})g(\gamma^{-1}(c_1)f(c_2)_{(1)})=\alpha(f(c_2))_{(0)}\alpha^{-1}(g(c_1\alpha(f(c_2)_{(1)}))),\end{equation}

     for all $f,g \in Hom^{\mathcal{H}}(C,A)$ and $c \in C$.
\end{enumerate}
\end{proposition}

\begin{proof}
\begin{enumerate}
\item By (\ref{Doi-Koppinen-entwining}) we have $a_{\kappa}\otimes \gamma(c)^{\kappa}=\alpha(a_{(0)})\otimes ca_{(1)}$, and thus

 \begin{eqnarray*}(aa')_{\kappa}\otimes \gamma(c)^{\kappa}&=&\alpha((aa')_{(0)})\otimes c((aa')_{(1)})\\
 &\overset{(\ref{Hom-comodule-algebra-cond})}{=}&\alpha(a_{(0)}a'_{(0)})\otimes c(a_{(1)}a'_{(1)})= \alpha(a_{(0)})\alpha(a'_{(0)})\otimes (\gamma^{-1}(c)a_{(1)})\beta(a'_{(1)})\\
 &\overset{(\ref{Doi-Koppinen-entwining})}{=}&a_{\kappa}\alpha(a'_{(0)})\otimes c^{\kappa}\beta(a'_{(1)})\\
 &=&a_{\kappa}\alpha(a'_{(0)})\otimes\gamma( \gamma^{-1}(c^{\kappa})a'_{(1)})\\
 &\overset{(\ref{Doi-Koppinen-entwining})}{=}&a_{\kappa}a'_{\lambda}\otimes \gamma(c^{\kappa\lambda}),
 \end{eqnarray*}

which shows that $\psi$ satisfies (\ref{hom-entwining-cond-1}). To prove that $\psi$ fulfills (\ref{hom-entwining-cond-2}) we have the computation

\begin{eqnarray*}\alpha^{-1}(a_{\kappa})\otimes c^{\kappa}_{\ 1}\otimes c^{\kappa}_{\ 2}&=&\alpha^{-1}(\alpha(a_{(0)}))\otimes (\gamma^{-1}(c)a_{(1)})_1\otimes (\gamma^{-1}(c)a_{(1)})_2\\
&\overset{(\ref{Hom-module-coalgebra-cond})}{=}&a_{(0)}\otimes \gamma^{-1}(c)_1a_{(1)1}\otimes\gamma^{-1}(c)_2a_{(1)2}\\
&=&a_{(0)}\otimes \gamma^{-1}(c_1)a_{(1)1}\otimes\gamma^{-1}(c_2)a_{(1)2}\\
&\overset{(\ref{Hom-coaction-cond-1})}{=}&\alpha(a_{(0)(0)})\otimes \gamma^{-1}(c_1)a_{(0)(1)}\otimes\gamma^{-1}(c_2)\beta^{-1}(a_{(1)})\\
&\overset{(\ref{Doi-Koppinen-entwining})}{=}&{a_{(0)}}_{\kappa}\otimes c_1^{\ \kappa}\otimes\gamma^{-1}(c_2)\beta^{-1}(a_{(1)})\\
&=&\alpha(\alpha^{-1}(a_{(0)}))_{\kappa}\otimes c_1^{\ \kappa}\otimes\gamma^{-1}(c_2)\beta^{-1}(a_{(1)})\\
&\overset{(\ref{Hom-coaction-cond-2})}{=}&\alpha(\alpha^{-1}(a)_{(0)})_{\kappa}\otimes c_1^{\ \kappa}\otimes\gamma^{-1}(c_2)\alpha^{-1}(a)_{(1)}\\
&\overset{(\ref{Doi-Koppinen-entwining})}{=}&\alpha^{-1}(a)_{\lambda\kappa}\otimes c_1^{\ \kappa}\otimes c_2^{\ \lambda}.
\end{eqnarray*}
To finish the proof of (1) we finally verify that $\psi$ satisfies (\ref{hom-entwining-cond-3}) and (\ref{hom-entwining-cond-4}) as follows,

$$1_{\kappa}\otimes c^{\kappa}=\alpha(1_{(0)})\otimes \gamma^{-1}(c)1_{(1)}=\alpha(1_A)\otimes \gamma^{-1}(c)1_B=1\otimes c,$$

\begin{eqnarray*}a_{\kappa}\varepsilon(c^{\kappa})&=&\alpha(a_{(0)})\varepsilon(\gamma^{-1}(c)a_{(1)})=\alpha(a_{(0)})\varepsilon(\gamma^{-1}(c\beta(a_{(1)})))\\
&\overset{(\ref{Hom-coring-cond-2})}{=}&\alpha(a_{(0)})\varepsilon(c\beta(a_{(1)}))\overset{(\ref{Hom-module-coalgebra-cond})}{=} \alpha(a_{(0)})\varepsilon(c)\varepsilon_B(\beta(a_{(1)})) \\
&=&\alpha(a_{(0)}\varepsilon_B(a_{(1)}))\varepsilon(c)\overset{(\ref{Hom-coaction-cond-1})}{=}\alpha(\alpha^{-1}(a))\varepsilon(c)\\
&=&a\varepsilon(c).
\end{eqnarray*}

\item We see that the condition for entwined Hom-modules,i.e., $\rho^{M}(ma)=m_{(0)}\alpha^{-1}(a)_{\kappa}\otimes \gamma(m_{(1)}^{\ \ \  \kappa})$ and the condition in (\ref{Doi-Koppinen-Hom-module-cond}) are equivalent by the following, for $m \in M$ and $a\in A$,

    \begin{eqnarray*}m_{(0)}\alpha^{-1}(a)_{\kappa}\otimes \gamma(m_{(1)}^{\ \ \  \kappa})&=&m_{(0)}\alpha(\alpha^{-1}(a)_{(0)})\otimes \gamma(\gamma^{-1}(m_{(1)})\alpha^{-1}(a)_{(1)})\\
    &=&m_{(0)}\alpha(\alpha^{-1}(a_{(0)}))\otimes\gamma(\gamma^{-1}(m_{(1)})\beta^{-1}(a_{(1)}))\\
    &=&m_{(0)}a_{(0)}\otimes\gamma(\gamma^{-1}(m_{(1)}a_{(1)}))\\
    &=&m_{(0)}a_{(0)}\otimes m_{(1)}a_{(1)}.
    \end{eqnarray*}

\item We only prove that the right $(A,\alpha)$-Hom-module structure holds as is given in the assertion. The rest of the structure of the corresponding Hom-coring can be seen at once from Theorem (\ref{Hom-Coring-Assoc-Hom-Entw-Strc}). For $a,a' \in A$ and $c\in C$,

     \begin{eqnarray*}(a'\otimes c)a&=&a'\alpha^{-1}(a)_{\kappa}\otimes \gamma(c^{\kappa})\\
     &=&a'\alpha(\alpha^{-1}(a)_{(0)})\otimes \gamma(\gamma^{-1}(c)\alpha^{-1}(a)_{(1)})=a'a_{(0)})\otimes \gamma(\gamma^{-1}(c)\beta^{-1}(a_{(1)}))\\
     &=&a'a_{(0)}\otimes ca_{(1)}.
      \end{eqnarray*}

\item By the definition of product $*_{\psi}$ given in Theorem (\ref{dual-algebra-of-Hom-coring-to-entw}) and the definition of $\psi$ given in (\ref{Doi-Koppinen-entwining}) we have, for $f,g \in Hom^{\mathcal{H}}(C,A)$ and $c \in C$,

\begin{eqnarray*}(f*_{\psi} g)(c)&=&f(c_2)_{\kappa}g(c_1^{\ \kappa})\\
&=&\alpha(f(c_2)_{(0)})g(\gamma^{-1}(c_1)f(c_2)_{(1)})=\alpha(f(c_2)_{(0)})g(\gamma^{-1}(c_1\beta(f(c_2)_{(1)})))\\
&=&\alpha(f(c_2)_{(0)})\alpha^{-1}(g(c_1\beta(f(c_2)_{(1)})))=\alpha(f(c_2))_{(0)}\alpha^{-1}(g(c_1\alpha(f(c_2))_{(1)})).
\end{eqnarray*}

\end{enumerate}
\end{proof}

\begin{definition} A triple $[(A,\alpha),(B,\beta),(C,\gamma)]$ is called a {\it (right-right) Hom-Doi-Koppinen datum} if it satisfies the conditions of Proposition (\ref{Hom-coring-assoc-Doi-Koppinen-datum}), that is, if $(A,\alpha)$ is a right $(B,\beta)$-Hom-comodule algebra and $(C,\gamma)$ is a right $(B,\beta)$-Hom-module coalgebra for a monoidal Hom-bialgebra $(B,\beta)$.

$[(A,\alpha), (C,\gamma)]_{\psi}$ in Proposition (\ref{Hom-coring-assoc-Doi-Koppinen-datum}) is called a {\it Hom-entwining structure associated to a Hom-Doi-Koppinen datum}.

A {\it Doi-Koppinen Hom-Hopf module} or a {\it unifying Hom-Hopf module} is a Hom-module satisfying the condition (\ref{Doi-Koppinen-Hom-module-cond}).
\end{definition}

Now we give the following collection of examples. Each of them is a special case of the construction given above.

\begin{example}{\textbf{Relative entwinings and relative Hom-Hopf modules}} Let $(B,\beta)$ be a monoidal Hom-bialgebra and let $(A,\alpha)$ be a $(B,\beta)$-Hom-comodule algebra with Hom-coaction $\rho^{A}:A\to A\otimes B, a\mapsto a_{(0)}\otimes a_{(1)}$.
\begin{enumerate}
\item $[(A,\alpha),(B,\beta)]_{\psi}$, with $\psi:B\otimes A\to A\otimes B,\: b\otimes a\mapsto \alpha(a_{(0)})\otimes \beta^{-1}(b)a_{(1)}$, is an Hom-entwining structure.
 \item $(M,\mu)$ is an $[(A,\alpha),(B,\beta)]_{\psi}$-entwined Hom-module if and only if it is a right $(A,\alpha)$-Hom-module with $\rho_{M}: M\otimes A \to M,\: m\otimes a\mapsto ma$ and a right $(B,\beta)$-Hom-comodule with $\rho^{M}:M\to M\otimes B,\: m\mapsto m_{[0]}\otimes m_{[1]}$ such that
   \begin{equation}\rho^{M}(ma)=m_{[0]}a_{(0)}\otimes m_{[1]}a_{(1)} \end{equation}
   for all $m\in M$ and $a\in A$. Hom-modules fulfilling the above condition are called {\it relative Hom-Hopf modules} (see \cite{GuoChen}).
 \item  $(\cC,\chi)=(A\otimes B,\alpha\otimes\beta)$ is a $(A,\alpha)$-Hom-coring with comultiplication $\Delta_{\cC}(a\otimes b)=(\alpha^{-1}(a)\otimes b_1)\otimes_A(1_A\otimes b_2)$ and counit $\varepsilon_{\cC}(a\otimes b)=\alpha(a)\varepsilon_B(b)$, and $(A,\alpha)$-Hom-bimodule structure $$a(a'\otimes b)=\alpha^{-1}(a)a'\otimes \beta(b),\:(a'\otimes b)a=a'a_{(0)}\otimes ba_{(1)}$$
    for all $a,a' \in A$ and $b \in B$.
\end{enumerate}

\end{example}

\begin{proof}The relevant Hom-Doi-Koppinen datum is $[(A,\alpha),(B,\beta),(B,\beta)]$, where the first object $(A,\alpha)$ is assumed to be a right $(B,\beta)$-Hom-comodule algebra with the Hom-coaction $\rho^{A}: a\mapsto a_{(0)}\otimes a_{(1)}$ and the third object $(B,\beta)$ is a right $(B,\beta)$-Hom-module coalgebra with Hom-action given by its Hom-multiplication. Hence, $[(A,\alpha),(B,\beta)]_{\psi}$ is the associated Hom-entwining structure, where $\psi(b\otimes a)=\alpha(a_{(0)})\otimes \beta^{-1}(b)a_{(1)}$. Assertions (2) and (3) can be seen at once from Proposition (\ref{Hom-coring-assoc-Doi-Koppinen-datum}).
\end{proof}

\begin{remark}$(A,\alpha)$ itself is a relative Hom-Hopf-module by its Hom-multiplication and the $(B,\beta)$-Hom-coaction $\rho^{A}$.
For $(A,\alpha)=(B,\beta)$ one gets the notion of {\it Hom-Hopf modules} (see \cite{CaenepeelGoyvaerts}).
\end{remark}

\begin{example}{\textbf{Dual-relative entwinings and} $[(C,\gamma),(A,\alpha)]$\textbf{-Hom-Hopf modules}} Let $(A,\alpha)$ be a monoidal Hom-bialgebra and let $(C,\gamma)$ be a right $(A,\alpha)$-Hom-module coalgebra with Hom-action $\rho_{C}:C\otimes A\to C, c \otimes a\mapsto ca$.
\begin{enumerate}
\item $[(A,\alpha),(C,\gamma)]_{\psi}$, with $\psi:C\otimes A\to A\otimes C,\: c\otimes a\mapsto \alpha(a_{1})\otimes \beta^{-1}(c)a_{2}$, is an Hom-entwining structure.
 \item $(M,\mu)$ is an $[(A,\alpha),(C,\gamma)]_{\psi}$-entwined Hom-module if and only if it is a right $(A,\alpha)$-Hom-module with $\rho_{M}: M\otimes A \to M,\: m\otimes a\mapsto ma$ and a right $(C,\gamma)$-Hom-comodule with $\rho^{M}:M\to M\otimes B,\: m\mapsto m_{(0)}\otimes m_{(1)}$ such that
   \begin{equation}\rho^{M}(ma)=m_{(0)}a_{1}\otimes m_{(1)}a_{2} \end{equation}
   for all $m\in M$ and $a\in A$. Such a Hom-module is called $[(C,\gamma),(A,\alpha)]${\it -Hom-Hopf module}.
 \item  $(\cC,\chi)=(A\otimes C,\alpha\otimes\gamma)$ is a $(A,\alpha)$-Hom-coring with comultiplication $\Delta_{\cC}(a\otimes c)=(\alpha^{-1}(a)\otimes c_1)\otimes_A(1_A\otimes c_2)$ and counit $\varepsilon_{\cC}(a\otimes b)=\alpha(a)\varepsilon_C(c)$, and $(A,\alpha)$-Hom-bimodule structure $$a(a'\otimes b)=\alpha^{-1}(a)a'\otimes \gamma(c),\:(a'\otimes c)a=a'a_{1}\otimes ca_{2}$$
    for all $a,a' \in A$ and $c \in C$.
\end{enumerate}
\end{example}

\begin{proof} $(A,\alpha)$ is a right $(A,\alpha)$-Hom-comodule algebra with Hom-coaction given by the Hom-comultiplication
$$\rho^{A}=\Delta_A: A \to A\otimes A,\: a\mapsto a_{(0)}\otimes a_{(1)}=a_1\otimes a_2,$$ since $\Delta_A$ is a Hom-algebra morphism. Besides $(C,\gamma)$ is assumed to be a right $(A,\alpha)$-Hom-module coalgebra with Hom-action $\rho_{C}(c\otimes a)=ca$. Thus, the related Hom-Doi-Koppinen datum is $[(A,\alpha),(A,\alpha),(C,\gamma)]$. Then $[(A,\alpha), (C,\gamma)]_{\psi}$ is the Hom-entwining structure associated to the datum, where
$$\psi(c\otimes a)=\alpha(a_{(0)})\otimes \gamma^{-1}(c)a_{(1)}=\alpha(a_{1})\otimes \gamma^{-1}(c)a_{2}.$$
The assertions (2) and (3) are also immediate by Proposition (\ref{Hom-coring-assoc-Doi-Koppinen-datum}).
\end{proof}
\begin{remark}$(C,\gamma)$ itself is a $[(C,\gamma),(A,\alpha)]$-Hom-Hopf-module by the $(A,\alpha)$-Hom-action $\rho_{C}$ and its Hom-comultiplication.
\end{remark}

The example below gives a Hom-generalization of the so-called $(\alpha,\beta)$-Yetter-Drinfeld modules introduced in \cite{PanaiteStaic} as an entwined Hom-module.

Following \cite{CaenepeelGoyvaerts}, a monoidal Hom-Hopf algebra $(H,\alpha)$ is a Hopf algebra in in $\widetilde{\mathcal{H}}(\mathcal{M}_k)$, i.e. it consists of
a septuple  $(H,\alpha,m,\eta,\Delta,\varepsilon,S)$ where $(H,\alpha,m,\eta,\Delta,\varepsilon)$ is a monoidal Hom-bialgebra and $S:H \to H$ is a morphism in $\widetilde{\mathcal{H}}(\mathcal{M}_k)$ such that $S\ast id_H=id_H\ast S=\eta\circ\varepsilon$. $S$ is called antipode and it has the following properties
\begin{equation*} S(gh)=S(h)S(g)\: ; \: S(1_H)=1_H\:;\Delta(S(h))=S(h_2)\otimes S(h_1)\: ; \: \varepsilon\circ S=\varepsilon,\end{equation*}
for any elements $g,h \in H$.

\begin{example}{\textbf{Generalized Yetter-Drinfeld entwinings and} $(\phi,\varphi)$\textbf{-Hom-Yetter-Drinfeld modules}}
Let $(H,\alpha)$ be a monoidal Hom-Hopf algebra and let $\phi, \varphi : H\to H$ be two monoidal Hom-Hopf algebra automorphisms. Define the map, for all $h,g\in H$
\begin{equation}\psi:H\otimes H\to H\otimes H, \: g\otimes h\mapsto \alpha^{2}(h_{21})\otimes \varphi(S(h_1))(\alpha^{-2}(g)\phi(h_{22})),\end{equation}
where $S$ is the antipode of $H$.
\begin{enumerate}
\item $[(H,\alpha),(H,\alpha)]_{\psi}$ is an Hom-entwining structure.
\item $(M,\mu)$ is an $[(H,\alpha),(H,\alpha)]_{\psi}$-entwined Hom-module if and only if it is a right $(H,\alpha)$-Hom-module with $\rho_{M}: M\otimes H \to M,\: m\otimes h\mapsto mh$ and a right $(H,\alpha)$-Hom-comodule with $\rho^{M}:M\to M\otimes H,\: m\mapsto m_{(0)}\otimes m_{(1)}$ such that
   \begin{equation}\label{generalized-YD-Hom-mod-cond}\rho^{M}(mh)=m_{(0)}\alpha(h_{21})\otimes \varphi(S(h_1))(\alpha^{-1}(m_{(1)})\phi(h_{22}))\end{equation}
   for all $m\in M$ and $h\in H$. A Hom-module $(M,\mu)$ satisfying this condition is called $(\phi,\varphi)$-{\it Hom-Yetter-Drinfeld module }.
\item $(\cC,\chi)=(H\otimes H, \alpha\otimes\alpha)$ is an $(H,\alpha)$-Hom-coring with comultiplication $\Delta_{\cC}(h\otimes h')=(\alpha^{-1}(h)\otimes h'_1)\otimes_H(1_H\otimes h'_2)$ and counit $\varepsilon_{\cC}(h\otimes h')=\alpha(h)\varepsilon_H(h')$, and $(H,\alpha)$-Hom-bimodule structure $$g(h\otimes h')=\alpha^{-1}(g)h\otimes \alpha(h'),\:(h\otimes h')g=h\alpha(g_{21})\otimes \varphi(S(g_1))(\alpha^{-1}(h')\phi(g_{22}))$$
    for all $h,h',g \in H$.
\end{enumerate}
\end{example}

\begin{proof} In the first place, we prove that the map $$\rho^{H}:H\to H\otimes(H^{op}\otimes H), \: h\mapsto h_{(0)}\otimes h_{(1)}:=\alpha(h_{21})\otimes (\alpha^{-1}(\varphi(S(h_1)))\otimes h_{22})$$ defines a $(H^{op}\otimes H,\alpha\otimes\alpha)$-Hom-comodule algebra structure on $(H,\alpha)$. Let us put $(H^{op}\otimes H,\alpha\otimes\alpha)=(\widetilde{H},\tilde{\alpha})$. Then

\begin{eqnarray*}\lefteqn{h_{(0)(0)}\otimes(h_{(0)(1)})\otimes \tilde{\alpha}^{-1}(h_{(1)})}\hspace{2em}\\
&=&\alpha(\alpha(h_{21})_{21})\otimes ((\alpha^{-1}(\varphi(S(\alpha(h_{21})_1)))\otimes \alpha(h_{21})_{22})\otimes (\alpha^{-2}(\varphi(S(h_1)))\otimes \alpha^{-1}(h_{22})))\\
&=&\alpha^{2}(h_{2121})\otimes ((\alpha^{-1}(\varphi(S(\alpha(h_{211}))))\otimes \alpha(h_{2122}))\otimes (\alpha^{-2}(\varphi(S(h_1)))\otimes \alpha^{-1}(h_{22})))\\
&=&\alpha^{2}(h_{2121})\otimes ((\varphi(S(h_{211}))\otimes \alpha(h_{2122}))\otimes (\alpha^{-2}(\varphi(S(h_1)))\otimes \alpha^{-1}(h_{22})))\\
&=&h_{21}\otimes ((\alpha^{-1}(\varphi(S(h_{12})))\otimes h_{221})\otimes (\alpha^{-1}(\varphi(S(h_{11})))\otimes h_{222}))\\
&=&h_{21}\otimes ((\alpha^{-1}(\varphi(S(h_{1})))_{1}\otimes h_{221})\otimes (\alpha^{-1}(\varphi(S(h_{1})))_{2}\otimes h_{222}))\\
&=&\alpha^{-1}(h_{(0)})\otimes \Delta_{\widetilde{H}}(h_{(1)}),
\end{eqnarray*}
where in the fourth step we used
$$\alpha(h_{11})\otimes \alpha^{-1}(h_{12})\otimes\alpha^{-2}(h_{21})\otimes\alpha^{-1}(h_{221})\otimes\alpha(h_{222})=h_{1}\otimes h_{211}\otimes h_{2121}\otimes h_{2122}\otimes h_{22},$$
which can be obtained by applying the Hom-coassociativity of $\Delta_H$ three times. We also have

\begin{eqnarray*}h_{(0)}\varepsilon_{\widetilde{H}}(h_{(0)})&=&\alpha(h_{21})\varepsilon(\alpha^{-1}(\varphi(S(h_{1}))))\varepsilon(h_{22})\\
&=&\alpha(h_{21}\varepsilon(h_{22}))\varepsilon(\alpha^{-1}(\varphi(S(h_{1}))))=\alpha(\alpha^{-1}(h_2))\varepsilon(h_1)\\
&=&\alpha^{-1}(h),
\end{eqnarray*}
where in the third equality we used the relations $\varepsilon\circ \alpha^{-1}=\varepsilon$, $\varepsilon\circ \varphi=\varepsilon$ and $\varepsilon\circ S=\varepsilon$. One can easily check that the relations $\rho^{H}\circ\alpha=(\alpha\otimes\tilde{\alpha})\circ\rho^{H}$ and $\rho^{H}(1_H)=1_H\otimes 1_{\widetilde{H}}$ hold.
For $g,h \in H$,

\begin{eqnarray*}\rho^{H}(g)\rho^{H}(g)&=&(\alpha(g_{21})\otimes (\alpha^{-1}(\varphi(S(g_1)))\otimes g_{22}))(\alpha(h_{21})\otimes (\alpha^{-1}(\varphi(S(h_1)))\otimes h_{22}))\\
&=&\alpha(g_{21})\alpha(h_{21})\otimes ( \alpha^{-1}(\varphi(S(h_1)))\alpha^{-1}(\varphi(S(g_1)))\otimes g_{22}h_{22})\\
&=&\alpha(g_{21}h_{21})\otimes ( \alpha^{-1}(\varphi(S(h_1)S(g_1)))\otimes g_{22}h_{22})\\
&=&\alpha((gh)_{21})\otimes ( \alpha^{-1}(\varphi(S((gh)_1)))\otimes (gh)_{22})\\
&=&\rho^{H}(gh),
\end{eqnarray*}
which completes the proof of the statement that $\rho^{H}$ makes $(H,\alpha)$ an $(\widetilde{H},\tilde{\alpha})$-Hom-comodule algebra. We next consider the map, for all $g,h,k \in H$
$$\rho_{H}: H\otimes \widetilde{H}\to H, \: g\cdot (h\otimes k):=(h\alpha^{-1}(g))\phi(\alpha(k))$$
and we claim that it defines an $(\widetilde{H},\tilde{\alpha})$-Hom-module coalgebra structure on $(H,\alpha)$: Indeed,

\begin{eqnarray*}(g\cdot(h\otimes k))\cdot(\alpha(h')\otimes\alpha(k'))&=&((h\alpha^{-1}(g))\phi(\alpha(k)))\cdot(\alpha(h')\otimes\alpha(k'))\\
&=&(\alpha(h')((\alpha^{-1}(h)\alpha^{-2}(g))\alpha^{-1}(\phi(\alpha(k)))))\phi(\alpha^{2}(k'))\\
&=&(\alpha(h')((\alpha^{-1}(h)\alpha^{-2}(g))\phi(k)))\phi(\alpha^{2}(k'))\\
&=&((h'(\alpha^{-1}(h)\alpha^{-2}(g)))\alpha(\phi(k)))\phi(\alpha^{2}(k'))\\
&=&(\alpha^{-1}((h'h)g)\alpha(\phi(k)))\phi(\alpha^{2}(k'))\\
&=&((h'h)g)(\alpha(\phi(k))\alpha^{-1}(\phi(\alpha^{2}(k'))))=((h'h)g)(\phi(\alpha(k))\phi(\alpha(k')))\\
&=&((h'h)g)\phi(\alpha((kk')))=((h'h)\alpha^{-1}(\alpha(g)))\phi(\alpha((kk')))\\
&=&\alpha(g)\cdot(h'h\otimes kk')=\alpha(g)\cdot((h\otimes k)(h'\otimes k')),
\end{eqnarray*}
$$h\cdot (1_H\otimes 1_H)=(1_H\alpha^{-1}(h))\phi(\alpha(1_H))=\alpha(h),$$

\begin{eqnarray*}(g\cdot(h\otimes k))_1\otimes(g\cdot(h\otimes k))_2&=& ((h\alpha^{-1}(g))\phi(\alpha(k)))_1\otimes ((h\alpha^{-1}(g))\phi(\alpha(k)))_2\\
&=&(h\alpha^{-1}(g))_1\phi(\alpha(k))_1\otimes (h\alpha^{-1}(g))_2\phi(\alpha(k))_2\\
&=&(h_1\alpha^{-1}(g_1))\phi(\alpha(k_1))\otimes (h_2\alpha^{-1}(g_2))\phi(\alpha(k_2))\\
&=&g_1\cdot(h_1\otimes k_1)\otimes g_2\cdot(h_2\otimes k_2)\\
&=&g_1\cdot(h\otimes k)_1\otimes g_2\cdot(h\otimes k)_2,
\end{eqnarray*}

$$\varepsilon(g\cdot(h\otimes k))=\varepsilon((h\alpha^{-1}(g))\phi(\alpha(k)))
=\varepsilon(h)\varepsilon(\alpha^{-1}(g))\varepsilon(\phi(\alpha(k)))=\varepsilon(h)\varepsilon(g)\varepsilon(k)
=\varepsilon(h)\varepsilon_{\widetilde{H}}(g\otimes k),$$
proving that $(H,\alpha)$ is an $(\widetilde{H},\tilde{\alpha})$-Hom-module coalgebra with the Hom-action $\rho_{H}$.
Hence, the Hom-Doi-Koppinen datum is given by $[(H,\alpha),(H^{op}\otimes H,\alpha\otimes\alpha),(H,\alpha)]$ to which the Hom-entwining structure $[(H,\alpha),(H,\alpha)]_{\psi}$ is associated, where we have the entwining map $\psi:H\otimes H\to H\otimes H$ as

\begin{eqnarray*}\psi(g\otimes h)&=&\alpha(h_{(0)})\alpha^{-1}(g)\cdot h_{(1)}=\alpha(\alpha(h_{21}))\otimes \alpha^{-1}(g)\cdot(\alpha^{-1}(\varphi(S(h_1)))\otimes h_{22})\\
&=&\alpha^{2}(h_{21})\otimes (\alpha^{-1}(\varphi(S(h_1)))\alpha^{-2}(g))\phi(\alpha(h_{22}))\\
&=&\alpha^{2}(h_{21})\otimes \varphi(S(h_1))(\alpha^{-2}(g)\phi(h_{22})).
\end{eqnarray*}

For $m\in M$ and $h\in H$, we have the condition (\ref{generalized-YD-Hom-mod-cond})

\begin{eqnarray*}\rho^{M}(mh)&=&m_{(0)}h_{(0)}\otimes m_{(1)}\cdot h_{(1)}\\
&=&m_{(0)}\alpha(h_{21})\otimes m_{(1)}\cdot(\alpha^{-1}(\varphi(S(h_1)))\otimes h_{22})\\
&=&m_{(0)}\alpha(h_{21})\otimes \alpha^{-1}(\varphi(S(h_1))m_{(1)})\phi(\alpha(h_{22}))\\
&=&m_{(0)}\alpha(h_{21})\otimes \varphi(S(h_1))(\alpha^{-1}(m_{(1)})\phi(h_{22})).
\end{eqnarray*}
By the above proposition, the $(H,\alpha)$-Hom-coring structure of $(H\otimes H,\alpha\otimes\alpha)$ is immediate. Here we only write down the right Hom-module condition
\begin{eqnarray*}(h\otimes h')g&=&hg_{(0)}\otimes h'\cdot g_{(1)}\\
&=&h\alpha(g_{21})\otimes h'\cdot (\alpha^{-1}(\varphi(S(g_1)))\otimes g_{22})\\
&=&h\alpha(g_{21})\otimes \varphi(S(g_1))(\alpha^{-1}(h') \phi(g_{22})),
\end{eqnarray*}
completing the proof.
\end{proof}

\begin{remark}
\begin{enumerate}
\item By putting $\phi=id_H=\varphi$ in the compatibility condition (\ref{generalized-YD-Hom-mod-cond}) we get the usual condition for (right-right) Hom-Yetter-Drinfeld modules, which is
    \begin{equation}\label{Hom-YD-mod-cond} \rho^{M}(mh)=m_{(0)}\alpha(h_{21})\otimes S(h_1)(\alpha^{-1}(m_{(1)})h_{22}).\end{equation}
 \item If the antipode $S$ of $(H,\alpha)$ is a bijection , then by taking $\phi=id_H$ and $\varphi=S^{-2}$ , we have the compatibility condition for (right-right) {\it anti-Hom-Yetter-Drinfeld modules} as follows
 \begin{equation}\label{anti-Hom-YD-mod-cond}\rho^{M}(mh)=m_{(0)}\alpha(h_{21})\otimes S^{-1}(h_1)(\alpha^{-1}(m_{(1)})h_{22}).\end{equation}
\end{enumerate}
\end{remark}

We get an equivalent condition for the generalized Hom-Yetter-Drinfeld modules by the following

\begin{proposition}The compatibility condition (\ref{generalized-YD-Hom-mod-cond}) for $(\phi,\varphi)$-Hom-Yetter-Drinfeld modules is equivalent to the equation \begin{equation}\label{equiv-generalized-YD-Hom-mod-cond}m_{(0)}\alpha^{-1}(h_1)\otimes m_{(1)}\phi(\alpha^{-1}(h_2))=(mh_2)_{(0)}\otimes \alpha^{-1}(\varphi(h_1)(mh_2)_{(1)}).\end{equation}
\end{proposition}

\begin{proof}Assume that (\ref{equiv-generalized-YD-Hom-mod-cond}) holds, then
\begin{eqnarray*} \lefteqn{m_{(0)}\alpha(h_{21})\otimes \varphi(S(h_1))(\alpha^{-1}(m_{(1)})\phi(h_{22}))}\hspace{4em}\\
&=&m_{(0)}\alpha^{-1}(\alpha^{2}(h_{21}))\otimes \varphi(S(h_1))(\alpha^{-1}(m_{(1)})\alpha^{-2}(\phi(\alpha^{2}(h_{22}))))\\
&=&m_{(0)}\alpha^{-1}(\alpha^{2}(h_2)_1)\otimes \varphi(S(h_1))\alpha^{-1}(m_{(1)}\alpha^{-1}(\phi(\alpha^{2}(h_2)_2)))\\
&\overset{(\ref{equiv-generalized-YD-Hom-mod-cond})}{=}&(m\alpha^{2}(h_2)_2)_{(0)}\otimes \varphi(S(h_1))(\alpha^{-2}(\varphi(\alpha^{2}(h_2)_1))\alpha^{-2}((m\alpha^{2}(h_2)_2)_{(1)}))\\
&=&(m\alpha^{2}(h_{22}))_{(0)}\otimes \varphi(S(h_1))(\varphi(h_{21})\alpha^{-2}((m\alpha^{2}(h_{22}))_{(1)}))\\
&\overset{(\ref{Hom-coring-cond-1})}{=}&(m\alpha(h_{2}))_{(0)}\otimes \varphi(S(\alpha(h_{11})))(\varphi(h_{12})\alpha^{-2}((m\alpha(h_{2}))_{(1)}))\\
&=&(m\alpha(h_{2}))_{(0)}\otimes \varphi(S(h_{11})h_{12})\alpha^{-1}((m\alpha(h_{2}))_{(1)})\\
&=&(m\alpha(h_{2}))_{(0)}\otimes \varphi(\varepsilon(h_1)1_H)\alpha^{-1}((m\alpha(h_{2}))_{(1)}))\\
&=&\varepsilon(h_1)(m\alpha(h_{2}))_{(0)}\otimes(m\alpha(h_{2}))_{(1)}\\
&=&\varepsilon(h_1)\rho^{M}(m\alpha(h_{2}))=\rho^{M}(mh),
\end{eqnarray*}
which gives us (\ref{generalized-YD-Hom-mod-cond}).
One can easily show that by applying the Hom-coassociativity condition (\ref{Hom-coring-cond-1}) twice we have
\begin{equation}\label{twice-Hom-coassociativity}\alpha^{-1}(h_1)\otimes h_{21}\otimes \alpha(h_{221})\otimes \alpha(h_{222})=h_{11}\otimes h_{12}\otimes h_{21}\otimes h_{22}, \end{equation}
which is used in the below computation.
 Thus, if we suppose that (\ref{generalized-YD-Hom-mod-cond}) holds, then
\begin{eqnarray*}\lefteqn{(mh_2)_{(0)}\otimes \alpha^{-1}(\varphi(h_1)(mh_2)_{(1)})}\hspace{3em}\\
&\overset{(\ref{generalized-YD-Hom-mod-cond})}{=}&m_{(0)}\alpha(h_{221})\otimes \alpha^{-1}(\varphi(h_1)(\varphi(S(h_{21}))(\alpha^{-1}(m_{(1)})\phi(h_{222}))))\\
&=&m_{(0)}\alpha(h_{221})\otimes \alpha^{-1}((\alpha^{-1}(\varphi(h_1))\varphi(S(h_{21})))(m_{(1)}\alpha(\phi(h_{222}))))\\
&\overset{(\ref{twice-Hom-coassociativity})}{=}&m_{(0)}h_{21}\otimes \alpha^{-1}((\varphi(h_{11})\varphi(S(h_{12})))(m_{(1)}\phi(h_{22})))\\
&=&m_{(0)}h_{21}\otimes (\varepsilon(h_1)1_H)\alpha^{-1}(m_{(1)}\phi(h_{22}))\\
&=&m_{(0)}h_{21}\otimes \varepsilon(h_1)m_{(1)}\phi(h_{22})\\
&\overset{(\ref{Hom-coring-cond-1})}{=}&m_{(0)}h_{12}\varepsilon(h_{11})\otimes m_{(1)}\phi(\alpha^{-1}(h_{2}))\\
&=&m_{(0)}\alpha^{-1}(h_1)\otimes m_{(1)}\phi(\alpha^{-1}(h_{2})),
\end{eqnarray*}
finishing the proof.
\end{proof}

\begin{remark} The above result implies that the equations (\ref{Hom-YD-mod-cond}) and (\ref{anti-Hom-YD-mod-cond}) are equivalent to
$$m_{(0)}\alpha^{-1}(h_1)\otimes m_{(1)}\alpha^{-1}(h_2)=(mh_2)_{(0)}\otimes \alpha^{-1}(h_1(mh_2)_{(1)})$$
and
$$m_{(0)}\alpha^{-1}(h_1)\otimes m_{(1)}\alpha^{-1}(h_2)=(mh_2)_{(0)}\otimes \alpha^{-1}(S^{-2}(h_1)(mh_2)_{(1)}),$$
respectively.
\end{remark}

\begin{example}{\textbf{The flip and Hom-Long dimodule}}
Let $(H,\alpha)$ be a monoidal Hom-bialgebra. Then:
\begin{enumerate}
\item $[(H,\alpha),(H,\alpha)]_{\psi}$, where $\psi:H\otimes H\to H\otimes H, \: g\otimes h\mapsto h\otimes g$, is an Hom-entwining structure.
\item $(M,\mu)$ is an $[(H,\alpha),(H,\alpha)]_{\psi}$-entwined Hom-module if and only if it is a right $(H,\alpha)$-Hom-module with $\rho_{M}: M\otimes H \to M,\: m\otimes h\mapsto mh$ and a right $(H,\alpha)$-Hom-comodule with $\rho^{M}:M\to M\otimes H,\: m\mapsto m_{(0)}\otimes m_{(1)}$ such that
   \begin{equation}\rho^{M}(mh)=m_{(0)}\alpha^{-1}(h)\otimes \alpha(m_{(1)})\end{equation}
   for all $m\in M$ and $h\in H$. Such Hom-modules $(M,\mu)$ are called {\it (right-right) $(H,\alpha)$-Hom-Long dimodules} (see \cite{ChenWangZhang1}).
\item $(\cC,\chi)=(H\otimes H, \alpha\otimes\alpha)$ is an $(H,\alpha)$-Hom-coring with comultiplication $\Delta_{\cC}(h\otimes h')=(\alpha^{-1}(h)\otimes h'_1)\otimes_H(1_H\otimes h'_2)$ and counit $\varepsilon_{\cC}(h\otimes h')=\alpha(h)\varepsilon_H(h')$, and $(H,\alpha)$-Hom-bimodule structure $$g(h\otimes h')=\alpha^{-1}(g)h\otimes \alpha(h'),\:(h\otimes h')g=h\alpha^{-1}(g)\otimes \alpha(h')$$
    for all $h,h',g \in H$.
\end{enumerate}
\end{example}

\begin{proof} $(H,\alpha)$ itself is a right $(H,\alpha)$-Hom-comodule algebra with Hom-coaction $\rho^{H}=\Delta_H:H\to H\otimes H,\:h\mapsto h_{(0)}\otimes h_{(1)}=h_1\otimes h_2$. In addition, $(H,\alpha)$ becomes a right $(H,\alpha)$-Hom-module coalgebra with the trivial Hom-action $\rho_{H}:H\otimes H\to H,\: g\otimes h\mapsto g\cdot h=\alpha(g)\varepsilon(h) $. Hence we have $[(H,\alpha),(H,\alpha),(H,\alpha)]$ as Hom-Doi-Koppinen datum with the associated Hom-entwining structure  $[(H,\alpha),(H,\alpha)]_{\psi}$, where $\psi(h'\otimes h)=\alpha(h_{(0)})\otimes \alpha^{-1}(h')\cdot h_{(1)}=\alpha(h_{1})\otimes \alpha^{-1}(h')\cdot h_{2}=\alpha(h_{1})\otimes \alpha(\alpha^{-1}(h'))\varepsilon(h_{2})=h\otimes h'$.
\end{proof}

\begin{definition}Let $(B,\beta)$ be a monoidal Hom-bialgebra. A {\it left} $(B,\beta)$-{\it Hom-comodule coalgebra} $(C,\gamma)$ is a monoidal Hom-coalgebra and a left $(B,\beta)$-Hom-comodule with a Hom-coaction $\rho:C\to B\otimes C,\: c \mapsto c_{(-1)}\otimes c_{(0)}$ such that, for any $c\in C$
\begin{equation}\label{Hom-comodule-coalgebra-cond}c_{(-1)}\otimes c_{(0)1}\otimes c_{(0)2}=c_{1(-1)}c_{2(-1)}\otimes c_{1(0)}\otimes c_{2(0)} ,\: c_{(-1)}\varepsilon_{C}(c_{(0)})=1_B\varepsilon_{C}(c),\end{equation}
$$\rho\circ\gamma=(\beta\otimes\gamma)\circ \rho.$$
\end{definition}

 We lastly introduce the below construction regarding the Hom-version of the so-called alternative Doi-Koppinen datum given in \cite{Schauenburg}.
For that we recall the definition of a Hom-module algebra from \cite{ChenWangZhang}: Let $(B,\beta)$ be a monoidal Hom-bialgebra. A {\it right} $(B,\beta)$-{\it Hom-module algebra} $(A,\alpha)$ is a monoidal Hom-algebra and a right $(B,\beta)$-Hom-module with a Hom-action $\rho_{A}: A\otimes B \to A,\: a\otimes b \mapsto a\cdot b$ such that, for any $a,a' \in A$ and $b\in B$
\begin{equation}\label{Hom-module-algebra-cond}b\cdot (aa')=(b_1\cdot a)(b_2\cdot a'),\:\: b\cdot 1_A=\varepsilon(b)1_A ,\end{equation}
$$\rho_{A}\circ (\alpha\otimes\beta)=\alpha\circ\rho_{A}.$$

\begin{proposition}\label{Alternative-Hom-Doi-Koppinen-datum}Let $(B,\beta)$ be a monoidal Hom-bialgebra. Let $(A,\alpha)$ be a left $(B,\beta)$-Hom-module algebra with Hom-action $_{A}\rho:B\otimes A \to A,\: b\otimes a\mapsto b\cdot a$ and $(C,\gamma)$ be a left $(B,\beta)$-Hom-comodule coalgebra with Hom-coaction $^{C}\rho:C\to B\otimes C,\: c\mapsto c_{(-1)}\otimes c_{(0)}$. Define the map
\begin{equation}\psi:C\otimes A\to A\otimes C,\: c\otimes a\mapsto c_{(-1)}\cdot \alpha^{-1}(a)\otimes \gamma(c_{(0)}) \end{equation}
Then the following statements hold.
\begin{enumerate}
\item $[(A,\alpha), (C,\gamma)]_{\psi}$ is an Hom-entwining structure.

\item $(M,\mu)$ is an $[(A,\alpha), (C,\gamma)]_{\psi}$-entwined Hom-module iff it is a right $(A,\alpha)$-Hom-module with $\rho_{M}:M\otimes A\to M,\: m\otimes a\mapsto ma$ and a right $(C,\gamma)$-Hom-comodule with $\rho^{M}:M\to M\otimes C,\: m\mapsto m_{[0]}\otimes m_{[1]}$ such that

\begin{equation}\label{alternative-Doi-Koppinen-Hom-module-cond}\rho^{M}(ma)=(ma)_{[0]}\otimes (ma)_{[1]}=m_{[0]}(m_{[1](-1)}\cdot\alpha^{-2}(a))\otimes \gamma^{2}(m_{[1](0)}) \end{equation}

for any $m\in M$ and $a\in A$.

\item $(\cC,\chi)=(A\otimes C, \alpha\otimes\gamma)$ is an $(A,\alpha)$-Hom-coring with comultiplication and counit given by (\ref{comult-of-associated-Hom-coring}) and (\ref{counit-of-associated-Hom-coring}), respectively, and the $(A,\alpha)$-Hom-bimodule structure $a(a'\otimes c)=\alpha^{-1}(a)a'\otimes \gamma(c)$, $(a'\otimes c)a=a'(c_{(-1)}\cdot \alpha^{-2}(a))\otimes\gamma^{2}(c_{(0)})$ for $a,a'\in A$ and $c\in C$.
\end{enumerate}
A triple $[(A,\alpha),(B,\beta),(C,\gamma)]$ satisfying the above assumptions of the proposition is called an {\it alternative Hom-Doi-Koppinen datum}.
\end{proposition}

\begin{proof} The first two conditions for Hom-entwining structures will be checked and the rest of the proof can be completed by performing similar computations as in Proposition (\ref{Hom-coring-assoc-Doi-Koppinen-datum}). For $a, a'\in A$ and $c\in C$,

\begin{eqnarray*}(aa')_{\kappa}\otimes \gamma(c)^{\kappa}&=&\gamma(c)_{(-1)}\cdot \alpha^{-1}(aa')\otimes \gamma(\gamma(c)_{(0)})\\
&=&\beta(c_{(-1)})\cdot (\alpha^{-1}(a)\alpha^{-1}(a'))\otimes \gamma^{2}(c_{(0)})\\
&=&(\beta(c_{(-1)})_1\cdot \alpha^{-1}(a))(\beta(c_{(-1)})_2\cdot \alpha^{-1}(a'))\otimes \gamma^{2}(c_{(0)})\\
&=&(\beta(c_{(-1)1})\cdot \alpha^{-1}(a))(\beta(c_{(-1)2})\cdot \alpha^{-1}(a'))\otimes \gamma^{2}(c_{(0)})\\
&=&(\beta(\beta^{-1}(c_{(-1)}))\cdot \alpha^{-1}(a))(\beta(c_{(0)(-1)})\cdot \alpha^{-1}(a'))\otimes \gamma^{2}(\gamma(c_{(0)(0)}))\\
&=&(c_{(-1)}\cdot \alpha^{-1}(a))(\gamma(c_{(0)})_{(-1)}\cdot \alpha^{-1}(a'))\otimes \gamma^{2}(\gamma(c_{(0)})_{(0)})\\
&=&(c_{(-1)}\cdot \alpha^{-1}(a))a'_{\lambda}\otimes \gamma(\gamma(c_{(0)})^{\lambda})\\
&=&a_{\kappa}a'_{\lambda}\otimes \gamma(c^{\kappa\lambda}),
\end{eqnarray*}
\begin{eqnarray*}\alpha^{-1}(a_{\kappa})\otimes c^{\kappa}_{\ 1}\otimes c^{\kappa}_{\ 2}&=&\alpha^{-1}(c_{(-1)}\cdot \alpha^{-1}(a))\otimes \gamma(c_{(0)})_1\otimes\gamma(c_{(0)})_2\\
&=&\beta^{-1}(c_{(-1)})\cdot \alpha^{-2}(a)\otimes \gamma(c_{(0)1})\otimes\gamma(c_{(0)2})\\
&=&\beta^{-1}(c_{1(-1)}c_{2(-1)})\cdot \alpha^{-2}(a)\otimes \gamma(c_{1(0)})\otimes\gamma(c_{2(0)})\\
&=&(\beta^{-1}(c_{1(-1)})\beta^{-1}(c_{2(-1)}))\cdot \alpha^{-2}(a)\otimes \gamma(c_{1(0)})\otimes\gamma(c_{2(0)})\\
&=&c_{1(-1)}\cdot(\beta^{-1}(c_{2(-1)})\cdot \alpha^{-3}(a))\otimes \gamma(c_{1(0)})\otimes\gamma(c_{2(0)})\\
&=&c_{1(-1)}\cdot\alpha^{-1}(c_{2(-1)}\cdot \alpha^{-2}(a))\otimes \gamma(c_{1(0)})\otimes\gamma(c_{2(0)})\\
&=&(c_{2(-1)}\cdot \alpha^{-1}(\alpha^{-1}(a)))_{\kappa}\otimes c_1^{\ \kappa}\otimes\gamma(c_{2(0)})\\
&=&\alpha^{-1}(a)_{\lambda\kappa}\otimes c_1^{\ \kappa}\otimes c_2^{\ \lambda}.
\end{eqnarray*}
\end{proof}

\section{Acknowledgments}
The author would like to thank Professor Christian Lomp for his valuable suggestions. This research was funded by the European Regional Development Fund through the programme COMPETE and by the Portuguese Government through the FCT- Fundação para a Ciência e a Tecnologia under the project PEst-C/MAT/UI0144/2013. The author was supported by the grant SFRH/BD/51171/2010.

\end{document}